\documentclass[12pt, reqno]{amsart}

\usepackage{amsmath, amsthm, amssymb}
\usepackage{enumerate}
\usepackage[margin=1.0in]{geometry}
\usepackage{xcolor}
\definecolor{cite}{rgb}{0.30,0.60,1.00}
\definecolor{url}{rgb}{0.00,0.00,0.80}
\definecolor{link}{rgb}{0.40,0.10,0.20}
\usepackage[pdfusetitle,colorlinks,linkcolor=link,urlcolor=url,citecolor=cite,pagebackref,breaklinks]{hyperref}
\usepackage{graphicx}
\usepackage{cleveref}
\usepackage{mathdots}
\usepackage{tikz-cd}
\usepackage{comment}
\usepackage{xypic}
\usepackage{mathtools}

\newtheorem{theorem}{Theorem}[section]

\newtheorem{proposition}[theorem]{Proposition}
\newtheorem{lemma}[theorem]{Lemma}

\newtheorem{corollary}[theorem]{Corollary}

\theoremstyle{definition}

\theoremstyle{definition}
\newtheorem{remark}[theorem]{Remark}
\theoremstyle{definition}
\newtheorem{example}[theorem]{Example}

\newcommand{\zIntegers}{\mathbb{Z}}

\newcommand{\cComplex}{\mathbb{C}}
\newcommand{\multiplicativegroup}[1]{#1^{\times}}

\newcommand{\Hom}{\mathrm{Hom}}

\newcommand{\idmap}{\mathrm{id}}
\newcommand{\conjugate}[1]{\overline{#1}}

\newcommand{\isomorphic}{\cong}
\newcommand{\lengthof}{\mathfrak{n}}
\newcommand{\abs}[1]{\left|#1\right|}
\newcommand{\sizeof}[1]{\left|#1\right|}
\newcommand{\lcm}{\operatorname{lcm}}
\newcommand{\partitionset}{\mathcal{P}}

\newcommand{\standardForm}[2]{\left\langle #1,#2\right\rangle}

\newcommand{\fieldCharacter}{\psi}
\newcommand{\centralCharacter}[1]{\omega_{#1}}
\newcommand{\Ind}[3]{\mathrm{Ind}_{#1}^{#2}\left(#3\right)}
\newcommand{\Whittaker}{\mathcal{W}}
\newcommand{\Contragradient}[1]{#1^{\vee}}

\newcommand{\representationDeclaration}[1]{#1}
\newcommand{\besselFunction}{\mathcal{J}}
\newcommand{\besselFunctionOfFiniteFieldRepresentation}{\besselFunction_{\finiteFieldRepresentation, \fieldCharacter}}

\newcommand{\lquot}[2]{{#1}\backslash{#2}}
\newcommand{\grpIndex}[2]{\left[#1:#2\right]}

\newcommand{\transpose}[1]{\, {}^{t}#1}

\newcommand{\IdentityMatrix}[1]{I_{#1}}
\newcommand{\diag}{\mathrm{diag}}

\newcommand{\trace}{\operatorname{tr}}
\newcommand{\GL}{\mathrm{GL}}
\newcommand{\UnipotentSubgroup}{U}

\newcommand{\FieldNorm}[2]{\mathrm{N}_{#1:#2}}
\newcommand{\FieldTrace}{\mathrm{Tr}}

\newcommand{\finiteFieldRepresentation}{\pi}
\newcommand{\finiteField}{\mathbb{F}}
\newcommand{\finiteFieldExtension}[1]{\finiteField_{#1}}
\newcommand{\FieldExtension}[2]{{#1} \slash {#2}}
\newcommand{\algebraicClosure}[1]{\overline{#1}}

\newcommand{\charactergroup}[1]{\widehat{\multiplicativegroup{\finiteFieldExtension{#1}}}}
\newcommand{\limitcharactergroup}{\Gamma}
\newcommand{\Galois}{\operatorname{Gal}}
\newcommand{\Frobenius}{\operatorname{Fr}}
\newcommand{\restrictionOfScalars}[3]{\operatorname{Res}_{#1 \slash #2}{#3}}
\newcommand{\multiplcativeScheme}{\mathbb{G}_m}
\newcommand{\affineLine}{\mathbb{A}^1}
\newcommand{\ladicnumbers}{\algebraicClosure{\mathbb{Q}_{\ell}}}
\newcommand{\artinScrier}{\operatorname{AS}_\fieldCharacter}
\newcommand{\convolutionWithCompactSupport}{\boldsymbol{\mathrm{R}}}
\newcommand{\etaleNorm}{\mathrm{N}}
\newcommand{\squareMatrix}{\operatorname{Mat}}
\newcommand{\leeYang}{\operatorname{LY}}
\newcommand{\frobeniusDegree}{\operatorname{deg}}
\newcommand{\Steinberg}{\operatorname{St}}

\hypersetup{pdfauthor={Elad Zelingher},
	pdfsubject={Representation theory},
	pdfkeywords={Bessel functions, Kloosterman sums, Kloosterman sheaves, Gauss sums}}

\title[Bessel function values for generic representations of $\GL_n$]{On values of the Bessel function for generic representations of finite general linear groups}

\author{Elad Zelingher}
\address{Department of Mathematics, University of Michigan, 1844 East Hall, 530 Church Street, Ann Arbor, MI 48109-1043 USA}
\email{eladz@umich.edu}


\begin{document}

\begin{abstract}
	We find a recursive expression for the Bessel function of S. I. Gelfand for irreducible generic representations of $\GL_n\left(\finiteField_q\right)$. We show that special values of the Bessel function can be realized as the coefficients of $L$-functions associated with exotic Kloosterman sums, and as traces of exterior powers of Katz's exotic Kloosterman sheaves. As an application, we show that certain polynomials, having special values of the Bessel function as their coefficients, have all of their roots lying on the unit circle. As another application, we show that special values of the Bessel function of the Shintani base change of an irreducible generic representation are related to special values of the Bessel function of the representation through Dickson polynomials.
\end{abstract}

\maketitle

\section{Introduction}\label{sec:introduction}

Let $\finiteField$ be a finite field with $q$ elements\footnote{In \cite{Gelfand70} it is assumed that $q \ne 2$, but our results are also true when $q = 2$.}, let $\fieldCharacter \colon \finiteField \rightarrow \multiplicativegroup{\cComplex}$ be a non-trivial additive character. For an irreducible generic representation $\finiteFieldRepresentation$ of $\GL_n\left(\finiteField\right)$, S. I. Gelfand defined its Bessel function $\besselFunctionOfFiniteFieldRepresentation$ in \cite[Section 4]{Gelfand70}, and gave a formula $$ \besselFunctionOfFiniteFieldRepresentation\left(g\right) = \frac{1}{\sizeof{\UnipotentSubgroup_n}} \sum_{u \in \UnipotentSubgroup_n}{\fieldCharacter^{-1}\left(u\right) \trace \left(\finiteFieldRepresentation \left(g u\right)\right) },$$
where $\UnipotentSubgroup_n$ is the standard unipotent subgroup of $\GL_n\left(\finiteField\right)$. S. I. Gelfand also described the support of the Bessel function. Computation of values of the Bessel function on its support using this formula is in general a hard task. Previous works include computations of the Bessel function for cuspidal representations of $\GL_2$ \cite{carter1992cuspidal}, generic representations of $\GL_2$ \cite{gel1962categories,piatetskishapirobook}, cuspidal representations of $\GL_3$ \cite{Gelfand70,carter1992cuspidal}, generic representations of $\GL_3$ \cite{helversen1982darstellungen}, cuspidal representations of $\GL_4$ \cite{gotsisphd1997, deriziotis1998cuspidal}, and generic representations of $\GL_4$ \cite{shinoda2005representations}. These computations are involved and it is a challenge to generalize them for larger $n$. A formula for the values of the Bessel function for $\left(n-1,1\right)$ anti-diagonal scalar block matrices for cuspidal representations of $\GL_n$ \cite{tulunay2004cuspidal, Nien17} and for generic representations of $\GL_n$ \cite{curtis2004zeta} is known. We also mention the recent preprint of Qu \cite{qu2022explicit} where he found a recursive formula for values of the Bessel function in the open Bruhat cell for generic principal series representations.

Bessel functions in their many different forms are central objects in representation theory of reductive groups. They have been used for many applications in the theory of automorphic forms, including the proof of functoriality by Cogdell--Kim--Piatetski-Shapiro--Shahidi \cite{cogdell2004functoriality}. Bessel distributions seem to have an important role in Langlands' proposal of ``Beyond Endoscopy'', see \cite{sakellaridis2022transfer}.

The Bessel function of S. I. Gelfand is closely related to gamma factors of the representations, see \cite{Roditty10,Nien14} for its relation to the Rankin--Selberg gamma factors, and \cite{YeZeligher18} for its relation to the exterior square gamma factors. The Bessel function served as a key ingredient in Nien's proof of Jacquet's conjecture \cite{Nien14} for finite fields. The ideas of Nien were later used by Chai for the proof of Jacquet's conjecture over $p$-adic fields \cite{chai2019bessel}, where the key ingredient is an analog of the Bessel function.

In a recent work \cite{ye2021epsilon}, Rongqing Ye and the author were able to express the Rankin--Selberg gamma factors of cuspidal irreducible representations explicitly as products of Gauss sums involving the multiplicative characters parameterizing the representations. In this paper, we are able to use the aforementioned result in order to find a recursive expression for the Bessel function, in terms of partitions of the sizes of the relevant anti-diagonal scalar block matrices. Our technique avoids computations of the trace characters of the representations, which avoids Green polynomials \cite{Green55} and conjugacy classes difficulties.

We are able to find a relation between special values of the Bessel function and Katz's exotic Kloosterman sums and sheaves. Kloosterman sheaves were introduced by Deligne in \cite{deligne569cohomologie} and studied extensively by Katz \cite{katz2016gauss}. They are used in order to bound exponential sums arising in analytic number theory, and have important applications \cite{castryck2014new, kowalski2017bilinear, kowalski2018stratification}. Our computation relates values of the Bessel function of the form $\besselFunctionOfFiniteFieldRepresentation\left( \begin{smallmatrix}
& \IdentityMatrix{n-m}\\ c \IdentityMatrix{m} \end{smallmatrix} \right)$ to the $L$-function attached to an exotic Kloosterman sum and to the Frobenius action on the $m$-th exterior power of a geometric stalk of a certain exotic Kloosterman sheaf. This is a surprising relation between representation theory of finite groups of Lie type and Kloosterman sums and sheaves. Our results can be stated as follows.

Let $\tau_1, \dots, \tau_r$ be irreducible cuspidal representations of $\GL_{n_1}\left(\finiteField\right), \dots, \GL_{n_r}\left(\finiteField\right)$, respectively, where $n_1 + \dots + n_r = n$. For every $1 \le i \le r$, let  $\alpha_i \colon \multiplicativegroup{\finiteFieldExtension{n_i}} \rightarrow \multiplicativegroup{\cComplex}$ be a regular multiplicative character associated with $\tau_i$, where $\finiteFieldExtension{n_i} \slash \finiteField$ is a field extension of degree $n_i$. Consider the representation $\pi$ of $\GL_n\left(\finiteField\right)$ defined as the unique irreducible generic subrepresentation of the parabolically induced representation $\tau_1 \circ \dots \circ \tau_r$. Any irreducible generic representation can be realized in this way.

Our first theorem is a generalization of the results of Curtis and Shinoda in \cite{curtis1999unitary}.
\begin{theorem}\label{thm:introduction-L-function}
	For any $a \in \multiplicativegroup{\finiteField}$, let $J_{m}\left(\alpha^{-1}, \fieldCharacter, a\right)$ be the following exotic Kloosterman sum of Katz.
	$$J_{m}\left(\alpha^{-1}, \fieldCharacter, a\right) = \sum_{\substack{x_1 \in \multiplicativegroup{(\finiteFieldExtension{n_1} \otimes_{\finiteField} \finiteFieldExtension{m})}, \dots, x_r \in \multiplicativegroup{(\finiteFieldExtension{n_r} \otimes_{\finiteField} \finiteFieldExtension{m} )}\\
	\prod_{i=1}^r \etaleNorm_{\finiteFieldExtension{n_i} \otimes \finiteFieldExtension{m} \slash \finiteFieldExtension{m}}\left(x\right) = a}} \left(\prod_{i=1}^r \alpha_i^{-1} \left(\etaleNorm_{\finiteFieldExtension{n_i} \otimes \finiteFieldExtension{m} \slash \finiteFieldExtension{n_i}}\left(x_i\right)\right) \right) \fieldCharacter\left(\sum_{i=1}^r \trace\left(x_i\right) \right).$$
Consider the $L$-function associated with this Kloosterman sum
$$L\left(T, J\left(\alpha^{-1}, \fieldCharacter, a\right)\right) = \exp \left(\sum_{m=1}^{\infty} J_{m}\left(\alpha^{-1}, \fieldCharacter, a\right) \frac{T^m}{m}\right)$$
and its normalized version $$L^{\ast}\left(T, J\left(\alpha^{-1}, \fieldCharacter, a\right) \right) = L\left(\left(-1\right)^r q^{-\frac{\left(n-1\right)}{2}} T, J\left(\alpha^{-1}, \fieldCharacter, a\right)\right)^{\left(-1\right)^n}.$$ Then
$$L^{\ast}\left(T, J\left(\alpha^{-1}, \fieldCharacter, \left(-1\right)^{n-1} a^{-1}\right)\right) = \sum_{m=0}^n  q^{\frac{m\left(n-m\right)}{2}} \besselFunction_{\pi, \fieldCharacter} \begin{pmatrix}
& \IdentityMatrix{n - m}\\
a \IdentityMatrix{m}
\end{pmatrix} T^m.$$
\end{theorem}
Our second theorem relates special values to Katz's exotic Kloosterman sheaves.
\begin{theorem}\label{thm:introduction-exterior-power}
	Let $\ell$ be a prime different than the characteristic of $\finiteField$. Fix an embedding $\ladicnumbers \hookrightarrow \cComplex$. Consider the following exotic Kloosterman sheaf of Katz $$\mathcal{K} = \operatorname{Kl}\left(\finiteFieldExtension{n_1} \times \dots \times \finiteFieldExtension{n_r}, \alpha^{-1}, \fieldCharacter\right) = \convolutionWithCompactSupport \mathrm{Norm}_{!} \left( \mathrm{Trace}^{\ast} \artinScrier \otimes \mathcal{L}_{\boxtimes_{i=1}^r \alpha_i^{-1}} \right) \left[n - 1\right],$$ associated with the diagram $$ \xymatrix{ & \restrictionOfScalars{(\finiteFieldExtension{n_1} \times \dots \times \finiteFieldExtension{n_r}) }{\finiteField}{\multiplcativeScheme} \ar[ld]_{ \mathrm{Norm} }  \ar[rd]^{ \mathrm{Trace} } \\
		\multiplcativeScheme & & \affineLine }.$$ Then for any $0 \le m \le n$, and any $c \in \multiplicativegroup{\finiteField},$
	$$ \left(\left(-1\right)^{r-1} q^{-\frac{\left(n-1\right)}{2}}\right)^m \trace \left( \Frobenius_{ \left(-1\right)^{n-1} c^{-1} } \mid \wedge^m \mathcal{K}_{ \left(-1\right)^{n-1} c^{-1} } \right) = q^{\frac{m \left(n-m\right)}{2}} \besselFunction_{\pi, \fieldCharacter}\begin{pmatrix}
	& \IdentityMatrix{n - m}\\
	c \IdentityMatrix{m}
	\end{pmatrix},$$
	where $\Frobenius_{ \left(-1\right)^{n-1} c^{-1} } \mid \wedge^m\mathcal{K}_{ \left(-1\right)^{n-1} c^{-1} }$ is the action of the geometric Frobenius at $\left(-1\right)^{n-1} c^{-1}$ acting on the $m$-th exterior power of the stalk of $\mathcal{K}$.
\end{theorem}
We remark that even though \Cref{thm:introduction-exterior-power} combined with Katz's theory of Kloosterman sheaves formally implies \Cref{thm:introduction-L-function}, our proof of \Cref{thm:introduction-L-function} is independent of the theory of Kloosterman sheaves.

These theorems show a deep connection between representation theory of $\GL_n\left(\finiteField\right)$ and Kloosterman sums and sheaves. For example, in a special case (\Cref{example:steinberg-usual-kloosterman}), the representation $\pi$ is the Steinberg representation $\Steinberg_n$ and we show that the normalized $L$-function of the usual Kloosterman sum is a polynomial whose coefficients are special values of the Bessel function of the Steinberg representation. This is a beautiful relation between the Steinberg representation and Kloosterman sums, both are fundamental objects in mathematics. Another special case (\Cref{example:principal-series-twisted-kloosterman}) establishes a relation between the Bessel function of generic principal series representations and twisted Kloosterman sums and sheaves.

This relation allows one to translate known results about Kloosterman sums and sheaves to the realm of representation theory of $\GL_n\left(\finiteField\right)$. For instance, Katz achieved equidistribution results for Kloosterman sums and a relation with the Sato--Tate measure \cite[Chapter 9]{katz2016gauss}. One could translate this into results about equidistribution of special values of the Bessel function of the Steinberg representation.

Although it is known that character sheaves are closely related to representation theory of finite groups of Lie type \cite{braverman2010gamma}, we are not aware of any direct relation as in the theorems above between Kloosterman sheaves and representation theory of finite groups of Lie type in the literature. We hope that this opens a door to more work on generalizations of this phenomenon, either for other groups, or for other geometric interpretations of values of the Bessel function for Bruhat cells with more than two anti-diagonal blocks. In the future, we hope to use similar ideas and techniques to find an expression for special values of the Bessel function of general supercuspidal representations of $p$-adic $\GL_n$ \cite[Section 5.1]{paskunas2008realization}.

As a result of the above mentioned relation, we are able to show that certain polynomials whose coefficients are special values of the Bessel function have the property that all of their roots lie on the unit circle. We also show that special values of the Bessel function of the Shintani base change of an irreducible generic representation are related to the special values of the Bessel function of the representation, by means of Dickson polynomials. We conclude by translating the converse theorem of Nien \cite{Nien14} to a statement about Kloosterman sheaves.

\subsection*{Acknowledgments}I would like to thank many people, as without their help and encouragement this work would have not been completed. I thank Rongqing Ye for his help at the beginning of this project, and for his comments on an early draft of this paper. I am grateful to David Soudry for sending me his unpublished proof of multiplicativity of gamma factors from 1979. I am indebted to Will Sawin for discussions about Kloosterman sheaves, for the observation of the relation with the exterior powers, and for the suggestion to consider the $L$-function. I thank Roger Howe and Mark Shusterman for useful discussions. I thank James Cogdell and Yifeng Liu for remarks on an early version of the paper. I am grateful to Ofir Gorodetsky for his encouragement along the way and for his comments on many revisions of this paper. Finally, I would like to thank the anonymous referee for their valuable comments and suggestions to improve the mathematical exposition of this paper. 

\section{Generic representations and their local constants}

In this section, we describe the irreducible generic representations of $\GL_{n}\left(\finiteField\right)$ in terms of Macdonald's parametrization \cite{Macdonald80}. Then we briefly review the theories of the Rankin--Selberg gamma factors, the intertwining operator gamma factors, and the tensor product $\epsilon_0$-factors for these representations. We also develop a recursive expression for the Bessel function, which will be used in \Cref{sec:formulas-for-special-values}.

\subsection{Macdonald's parametrization of irreducible representations of $\GL_n\left(\finiteField\right)$}

We briefly review the parametrization of irreducible representations of $\GL_n\left(\finiteField\right)$, as described in \cite[Section 1]{Macdonald80}.

Let $\finiteField$ be a field with $q$ elements. Let $\algebraicClosure{\finiteField}$ be an algebraic closure of $\finiteField$. For each positive integer $d$, we denote by $\finiteFieldExtension{d}$ the (unique) field extension of $\finiteField$ of degree $d$ in $\algebraicClosure{\finiteField}$. We denote by  $\multiplicativegroup{\finiteFieldExtension{d}}$ the multiplicative group of $\finiteFieldExtension{d}$, and by $\charactergroup{d}$ the character group of $\multiplicativegroup{\finiteFieldExtension{d}}$. A multiplicative character $\alpha \in \charactergroup{d}$ is called regular if the set $\{ \alpha, \alpha^q, \dots, \alpha^{q^{d-1}} \}$ is of size $d$.

For each $d' \mid d$, we have the norm map $\FieldNorm{d}{d'} \colon \multiplicativegroup{\finiteFieldExtension{d}} \rightarrow \multiplicativegroup{\finiteFieldExtension{d'}}$. These maps induce maps $\widehat{\FieldNorm{d}{d'}} \colon \charactergroup{d'} \rightarrow \charactergroup{d}$, by mapping $\gamma \in \charactergroup{d'}$ to $\gamma \circ \FieldNorm{d}{d'} \in \charactergroup{d}$. We have that $(\charactergroup{d})_{d}$ with the norm maps $(\widehat{\FieldNorm{d}{d'}})_{d' \mid d}$ forms a directed system. We denote its direct limit $\limitcharactergroup = \varinjlim \charactergroup{d}$.

Let $\Frobenius \in \Galois \left( \FieldExtension{\algebraicClosure{\finiteField}}{\finiteField} \right)$ be the geometric Frobenius automorphism, i.e., $\Frobenius\left(x^q\right) = x$, for every $x \in \algebraicClosure{\finiteField}$. Then $\Frobenius$ acts on $\limitcharactergroup$ by $\Frobenius \gamma = \gamma^q$. We identify $\charactergroup{d}$ with the following subgroup of $\limitcharactergroup$ $$\limitcharactergroup_d = \left\{ \gamma \in \limitcharactergroup \mid \Frobenius^d\gamma = \gamma \right\}.$$ A Frobenius orbit is a set of the form $$f=\left\{ \Frobenius^i \gamma \mid i \in \zIntegers  \right\},$$ where $\gamma \in \Gamma$. Given a Frobenius orbit $f$, we define its degree $\frobeniusDegree \left(f\right)$ to be the cardinality of $f$. Then for $\gamma \in f$, we have $\gamma \in \limitcharactergroup_{\frobeniusDegree \left(f\right)}$. We denote by $\lquot{\Frobenius}{\Gamma}$ the set of Frobenius orbits.

A partition of an integer $n \ge 0$, is a tuple of integers $\lambda = \left(n_1, \dots, n_r \right)$, with $n_1 \ge n_2 \ge \dots \ge n_r > 0$, and such that $n_1 + \dots + n_r = n$. We denote by $\sizeof{\lambda} = n$, the size of $\lambda$, we denote by $\lengthof\left(\lambda\right) = r$, the length of $\lambda$. We write $\lambda \vdash n$ to specify that $\lambda$ is a partition of $n$. If $d$ is a positive integer, we write $d \lambda$ for the partition $\left(d n_1, d n_2, \dots, d n_r\right)$ of $dn$. We denote by $\left(\right)$ the empty partition of $0$. Let $\partitionset$ be the set of all partitions of all non-negative integers.

We denote by $P_n\left(\limitcharactergroup\right)$ the set of partition valued functions $\phi \colon \Gamma \rightarrow \partitionset$, such that:
\begin{enumerate}
	\item $\phi \circ \Frobenius = \phi$, i.e., $\phi$ is constant on Frobenius orbits.
	\item $\sum_{\gamma \in \limitcharactergroup} {\sizeof{\phi \left(\gamma\right)}} = n$. 
\end{enumerate}
If $f$ is a Frobenius orbit, then such $\phi$ is constant on the elements of $f$, and we define $\phi \left(f\right) = \phi \left(\gamma\right)$, where $\gamma \in f$.

For representations $\pi_1, \dots, \pi_r$ of $\GL_{n_1}\left(\finiteField\right), \dots, \GL_{n_r}\left(\finiteField\right)$, respectively, we denote their parabolic induction, a representation of $\GL_{n_1 + \dots + n_r}\left(\finiteField\right)$, by $\pi_1 \circ \dots \circ \pi_r$. This operation is commutative and associative.

We now describe the irreducible representations of $\GL_n \left(\finiteField\right)$. These are in a bijection with the set $P_n \left(\limitcharactergroup\right)$. First, let $f$ be a Frobenius orbit of degree $d$, then $f$ corresponds to an irreducible cuspidal representation $\Pi_{f}$ of $\GL_{d}\left(\finiteField\right)$. For any positive integer $s$, consider the representation of $\GL_{ds}$ defined by $\Pi_{f}^{\circ s} = \Pi_{f} \circ \dots \circ \Pi_{f}$, where $\circ$ is performed $s$ times. Then the irreducible subrepresentations of $\Pi_{f}^{\circ s}$ are indexed by partitions of $s$, see \cite[Appendices B and C]{gurevich2021harmonic}. For each partition $\lambda$ of $s$, we denote by $\Pi_f^{\lambda}$ the irreducible subrepresentation of $\Pi_f^{\circ s}$ corresponding to $\lambda$.

Let $\phi \in P_n \left(\Gamma\right)$. Suppose that $\phi$ is supported on the pairwise distinct Frobenius orbits $f_1,\dots,f_t$, i.e., $\phi \left(\gamma\right) \ne \left(\right)$ if and only if $\gamma \in f_i$ for some $1 \le i \le t$. We associate to $\phi$ the representation $\Pi_{\phi} = \Pi_{f_1}^{\phi \left(f_1\right)} \circ \dots \circ \Pi_{f_t}^{\phi \left( f_t \right)}$.

The following theorem is stated in \cite[Section 1]{Macdonald80}.

\begin{theorem}\label{thm:greens-parameterization-of-representations}
	The map $\phi \mapsto \Pi_{\phi}$ is a bijection between $P_n\left(\limitcharactergroup\right)$ and the set of equivalence classes of irreducible representations of $\GL_{n} \left(\finiteField\right)$.
\end{theorem}

Let $\alpha_1 \in \charactergroup{\frobeniusDegree\left(f_1
	\right)}$, $\dots$, $\alpha_t \in \charactergroup{\frobeniusDegree\left(f_t
	\right)}$ be representatives for the Frobenius orbits $f_1,\dots,f_t$, respectively. Then the central character of $\Pi_\phi$ is given by
$$ \centralCharacter{\Pi_\phi}\left(z\right) = \prod_{i=1}^t \alpha_i \left(z\right),$$ where $z \in \multiplicativegroup{\finiteField}$.

\subsection{Parametrization of irreducible generic representations of $\GL_n \left(\finiteField\right)$}

Let $\fieldCharacter \colon \finiteField \rightarrow \multiplicativegroup{\cComplex}$ be a non-trivial additive character. Let $\UnipotentSubgroup_n$ be the upper unipotent subgroup of $\GL_{n} \left(\finiteField\right)$. Then $\fieldCharacter$ defines a character on $\UnipotentSubgroup_n$ by $$\fieldCharacter\begin{pmatrix}
1 & a_1 & \ast & \ast & \ast\\
& 1 & a_2 & \ast & \ast\\
& & \ddots & \ddots & \ast \\
& &  & 1 & a_{n-1} \\
& & & & 1
\end{pmatrix} = \fieldCharacter\left(\sum_{i=1}^{n-1}{a_i}\right).$$

A representation $\representationDeclaration{\pi}$ of $\GL_{n}\left(\finiteField\right)$ is called generic if $\Hom_{U_n}\left(\pi\restriction_{\UnipotentSubgroup_n},\fieldCharacter\right) \ne 0$, or equivalently by Frobenius reciprocity, if $\Hom_{\GL_{n}\left(\finiteField\right)}(\pi, \Ind{\UnipotentSubgroup_n}{\GL_{n}\left(\finiteField\right)}{\fieldCharacter}) \ne 0$. This condition does not depend on the choice of the non-trivial additive character $\fieldCharacter$. If $\pi$ is irreducible and generic, then the latter Hom space is one dimensional \cite[Corollary 5.6]{SilbergerZink00}.

It is known that if $\pi_1, \dots, \pi_r$ are irreducible representations of $\GL_{n_1}\left(\finiteField\right), \dots, \GL_{n_r}\left(\finiteField\right)$, respectively, then $\pi_1 \circ \dots \circ \pi_r$ is generic, if and only if $\pi_1, \dots, \pi_r$ are all generic, see \cite[Theorem 5.5]{SilbergerZink00} (they only state one direction, but the other direction follows from their proof). Therefore, in light of \Cref{thm:greens-parameterization-of-representations}, in order to classify the irreducible generic representations, it suffices to classify the irreducible generic subrepresentations of $\Pi_f^{\circ s}$, for a Frobenius orbit $f$ and $s \ge 1$.

It is known that irreducible cuspidal representations of $\GL_n\left(\finiteField\right)$ are generic \cite[Lemma 5.2]{SilbergerZink00}. By \cite{SilbergerZink00}, if $f$ is a Frobenius orbit of degree $d$, then for any $s$, the representation $\Pi_f^{\circ s}$ has a unique irreducible generic subrepresentation, which corresponds to the partition $\left(s\right)$ of $s$. This is the ``generalized Steinberg" representation. By \cite[Page 3346, Equation (5)]{SilbergerZink00}, its dimension is given by the formula $$\dim \Pi_f^{\left(s\right)} = q^{\frac{d s \left(s - 1\right)}{2}} \frac{\prod_{j=1}^{d s}{\left(q^j - 1\right)}}{\prod_{j=1}^{s}{\left(q^{d j} - 1\right)}}.$$

Therefore we have that irreducible generic representations are parameterized by $\phi \in P_n\left(\limitcharactergroup\right)$, such that for every $\gamma \in \limitcharactergroup$, $\phi \left( \gamma \right) = \left(\right)$ or $\phi \left( \gamma \right)$ is of the form $\left(s\right)$, where $s$ is a positive integer.

Suppose that $\phi \in P_n\left(\limitcharactergroup\right)$ parameterizes an irreducible generic representation as above and is supported on the Frobenius orbits $f_1, \dots, f_t$ of degrees $d_1 = \frobeniusDegree\left(f_1\right), \dots, d_t=\frobeniusDegree\left(f_t\right)$, respectively, and that for every $i$, $\phi \left(f_i\right) = \left({s_i}\right)$. Then $$\dim \Pi_{\phi} = \frac{\sizeof{\GL_{n} \left(\finiteField\right)}}{\abs{P_{d_1 s_1,\dots, d_t s_t}}} \cdot \prod_{i=1}^{t}{\dim{\Pi_{f_i}^{\left({s_i}\right)}}},$$
where $P_{d_1 s_1,\dots,d_t s_t} \subseteq \GL_n\left(\finiteField\right)$ is the standard parabolic subgroup corresponding to the composition $\left(d_1 s_1,\dots, d_t s_t\right)$. Let $N_{d_1 s_1,\dots, d_t s_t}$ be the unipotent radical of $P_{d_1 s_1,\dots, d_t s_t}$. We have that $$\sizeof{P_{d_1 s_1,\dots,d_t s_t}} = \sizeof{N_{d_1 s_1,\dots, d_t s_t}} \cdot \prod_{i=1}^{t}{\sizeof{\GL_{d_i s_i}\left(\finiteField\right)}} = \frac{\sizeof{\UnipotentSubgroup_n}}{\prod_{i=1}^{t}\sizeof{\UnipotentSubgroup_{d_i s_i}}} \cdot \prod_{i=1}^{t}{\sizeof{\GL_{d_i s_i}\left(\finiteField\right)}}.$$ Therefore, by using the formula $\sizeof{\GL_{ds}\left(\finiteField\right)} = \sizeof{\UnipotentSubgroup_{ds}} \cdot \prod_{j=1}^{ds}\left(q^j - 1\right)$, we get
\begin{equation}\label{eq:dimension-of-generic-representation}
	\dim \Pi_{\phi} = \grpIndex{\GL_{n}\left(\finiteField\right)}{\UnipotentSubgroup_n} \cdot \prod_{i=1}^{t}\frac{q^{\frac{d_i s_i \left(s_i - 1\right)}{2}}}{\prod_{j=1}^{s_i}\left(q^{d_i j} - 1\right)}.
\end{equation}

We note that one can give an alternative definition for $\Pi_\phi$ by defining it to be the unique irreducible generic subrepresentation of $\Pi_{f_1}^{\circ s_1}\circ \dots \circ \Pi_{f_r}^{\circ s_r}$.

Recall that irreducible cuspidal representations of $\GL_{n}\left(\finiteField\right)$ are in a (non-canonical) bijection with monic irreducible polynomials of degree $n$, not having zero as a root. By factorization of polynomials of degree $n$ into powers of irreducible polynomials, we get that the irreducible generic representations of $\GL_{n}\left(\finiteField\right)$ are in a bijection with monic polynomials that do not vanish at zero. Therefore there exist exactly $q^{n} - q^{n-1}$ irreducible generic representations of $\GL_{n}\left(\finiteField\right)$.

\subsection{Whittaker models and their Bessel functions}
Recall that the representation $\Ind{\UnipotentSubgroup_n}{\GL_{n}\left(\finiteField\right)}{\fieldCharacter}$ is multiplicity free \cite[Theorem 0.5]{Gelfand70}, \cite[Corollary 5.6]{SilbergerZink00}, i.e., for every irreducible representation $\representationDeclaration{\pi}$ of $\GL_{n}\left(\finiteField\right)$, we have that $$\dim \Hom_{\GL_{n}\left(\finiteField\right)}(\pi, \Ind{\UnipotentSubgroup_n}{\GL_n\left(\finiteField\right)}{\fieldCharacter}) \le 1.$$ If $\pi$ is irreducible and generic, we get that $\dim \Hom_{\GL_{n}\left(\finiteField\right)} (\pi, \Ind{\UnipotentSubgroup_n}{\GL_n\left(\finiteField\right)}{\fieldCharacter}) = 1$, and we denote by $\Whittaker\left(\pi, \fieldCharacter\right)$ the unique subspace of $\Ind{\UnipotentSubgroup_n}{\GL_{n}\left(\finiteField\right)}{\fieldCharacter}$ that is isomorphic to $\pi$, that is the \emph{Whittaker model of $\pi$, with respect to the additive character $\fieldCharacter$}.

Suppose that $\representationDeclaration{\pi}$ is an irreducible generic representation of $\GL_{n}\left(\finiteField\right)$. There exists a special element of $\Whittaker\left(\pi,\fieldCharacter\right)$, denoted by $\besselFunction_{\pi, \fieldCharacter}$, which is called the \emph{normalized Bessel function of $\pi$, with respect to the additive character $\fieldCharacter$}. It is the unique element of $\Whittaker\left(\pi, \fieldCharacter\right)$ satisfying the following properties:
\begin{enumerate}
	\item $\besselFunction_{\pi, \fieldCharacter}\left( \IdentityMatrix{n} \right) = 1$.
	\item $\besselFunction_{\pi, \fieldCharacter}\left( u_1 g u_2 \right) = \fieldCharacter\left(u_1\right)\fieldCharacter\left(u_2\right)\besselFunction_{\pi, \fieldCharacter}\left(g\right)$, for every $u_1,u_2 \in \UnipotentSubgroup_n$ and $g \in \GL_{n}\left(\finiteField\right)$.
\end{enumerate}
We note that for $z \in \multiplicativegroup{\finiteField}$ and $g \in \GL_n\left(\finiteField\right)$, we have $\besselFunction_{\pi, \fieldCharacter}\left(zg\right) = \centralCharacter{\pi}\left(z\right) \besselFunction_{\pi, \fieldCharacter}\left(g\right)$, where $\centralCharacter{\pi}$ is the central character of $\pi$.

We have the following proposition regarding the support of $\besselFunction_{\pi, \fieldCharacter}$.
\begin{proposition}[{\cite[Proposition 4.9]{Gelfand70}}]\label{prop:support-of-bessel-function}
	$\besselFunction_{\pi, \fieldCharacter}$ is supported on double cosets of the form $$\UnipotentSubgroup_n \cdot g_{n_1,\dots,n_s}\left(c_1,\dots,c_s\right)  \cdot \UnipotentSubgroup_n,$$ where 	
for $n_1,\dots,n_s > 0$ with $n = n_1 + \dots + n_s$ and $c_1, \dots, c_s \in \multiplicativegroup{\finiteField}$, $$g_{n_1,\dots,n_s}\left(c_1,\dots,c_s\right) = \begin{pmatrix}
& & & c_1 \IdentityMatrix{n_1}\\
& & c_2 \IdentityMatrix{n_2}\\
& \iddots\\
c_s \IdentityMatrix{n_s}
\end{pmatrix}.$$
\end{proposition}

A counting argument shows that there exist exactly $q^n - q^{n-1}$ options for a matrix of the form $g_{n_1,\dots,n_s}\left(c_1,\dots,c_s\right) \in \GL_{n}\left(\finiteField\right)$, for any choice of $s$, $n_1,\dots,n_s > 0$ with $n_1+\dots+n_s = n$ and any choice of $c_1,\dots,c_s \in \multiplicativegroup{\finiteField}$. Indeed, these are in bijection with sequences $\left(a_i\right)_{i=1}^n \subseteq \finiteField$, such that $a_1 \ne 0$. The element $0$ represents continuation of the previous block, while any other element represents the beginning of a new block with the given symbol. For example, the sequence $\left(c_1,0,0,0,c_2,c_3,0,c_4\right)$, where $c_1,\dots,c_4 \in \multiplicativegroup{\finiteField }$ represents $g_{4,1,2,1}\left(c_1,c_2,c_3,c_4\right) \in \GL_8 \left(\finiteField\right)$.

We also have the following relations between $\besselFunction_{\pi, \fieldCharacter}$ and its complex conjugate. \begin{proposition}[{\cite[Proposition 2.15, Proposition 3.5]{Nien14}}]\label{prop:complex-conjugate-of-bessel-function}For any $g \in \GL_n\left(\finiteField\right)$:
	\begin{enumerate}
		\item $\besselFunction_{\pi, \fieldCharacter}\left(g^{-1}\right) = \conjugate{\besselFunction_{\pi, \fieldCharacter}\left(g\right)}$.
		\item $\besselFunction_{\Contragradient{\pi}, \fieldCharacter^{-1}}\left(g\right) = \conjugate{\besselFunction_{\pi, \fieldCharacter}\left(g\right)}$, where $\Contragradient{\pi}$ denotes the contragredient representation of $\pi$.
	\end{enumerate}
\end{proposition}
S. I. Gelfand \cite[Proposition 4.5]{Gelfand70} gives a formula for $\besselFunction_{\pi, \fieldCharacter}$ in terms of the trace character of the representation $\pi$.
\begin{theorem}\label{thm:bessel-function-as-sum-of-trace}For any $g \in \GL_n\left(\finiteField\right)$,
	$$\besselFunction_{\pi, \fieldCharacter}\left(g\right) = \frac{1}{\sizeof{\UnipotentSubgroup_n}} \sum_{u \in \UnipotentSubgroup_n}{\fieldCharacter^{-1}\left(u\right) \trace\left(\pi\left(g u\right)\right)}.$$
\end{theorem}

\subsection{Rankin--Selberg gamma factors}
Let $n \ge m$ be integers and let $\pi$ and $\sigma$ be irreducible generic representations of $\GL_n\left(\finiteField\right)$ and $\GL_m\left(\finiteField\right)$, respectively.
The Rankin--Selberg gamma factor $\gamma\left(\pi\times\sigma,\fieldCharacter\right)$ was defined in Piatetski-Shapiro's unpublished lecture notes from 1976. It was also defined for $m < n$ in Roditty-Gershon's master's thesis, under the supervision of David Soudry \cite{Roditty10}. The main ideas of Roditty-Gershon's thesis are covered by Nien in \cite{Nien14}. We briefly review the main results that we need.

The first result is the functional equation, which defines the Rankin--Selberg gamma factors \cite[Theorem 2.10]{Nien14}.

\begin{theorem}
	Suppose that $n > m$, and that $\pi$ is cuspidal. Then there exists a non-zero constant $\gamma\left( \pi \times \sigma, \fieldCharacter \right) \in \multiplicativegroup{\cComplex}$, such that for every $0 \le k \le n-m-1$, every $W \in \Whittaker\left(\pi, \fieldCharacter\right)$ and $W' \in \Whittaker\left(\sigma, \fieldCharacter^{-1}\right)$, the following functional equation holds:
	
	\begin{align*}
		& q^{m k} \gamma\left(\pi \times \sigma, \fieldCharacter \right) \sum_{h\in \lquot{\UnipotentSubgroup_m}{\GL_{m}\left(\finiteField\right)}}\sum_{x \in M_{ \left(n - m - k - 1\right) \times m}\left(\finiteField\right)}{W \begin{pmatrix}
			h &\\
			x & \IdentityMatrix{n-m-k-1} \\
			& & \IdentityMatrix{k+1}			
			\end{pmatrix}}W'\left(h\right) \\
		= & \sum_{h\in \lquot{\UnipotentSubgroup_m}{\GL_{m}\left(\finiteField\right)}}\sum_{x \in M_{m \times k}\left(\finiteField\right)} W \begin{pmatrix}
		& \IdentityMatrix{n-m-k} \\
		& & \IdentityMatrix{k}\\
		h & & x
		\end{pmatrix}
		W'\left(h\right).		
	\end{align*}
\end{theorem}

Piatetski-Shapiro proved in his lecture the functional equation for the case $n = m$ \cite[Theorem 2.3]{Ye18}.
\begin{theorem}
	Suppose that $n=m$, and that $\pi$ and $\sigma$ are cuspidal. There exists a non-zero constant $\gamma\left(\pi\times\sigma,\fieldCharacter\right)$, such that for every $W \in \Whittaker\left(\pi,\fieldCharacter\right)$ and $W' \in \Whittaker\left(\sigma
	, \fieldCharacter^{-1}\right)$, and for every $\phi \colon \finiteField^n \rightarrow \cComplex$ with $\phi \left(0\right) = 0$, the following functional equation holds:
	$$ \gamma\left(\pi\times\sigma, \fieldCharacter\right)\sum_{g \in \lquot{\UnipotentSubgroup_n}{\GL_{n}\left(\finiteField\right)}} W\left(g\right)W'\left(g\right)\phi\left(e_n g\right) = \sum_{g \in \lquot{\UnipotentSubgroup_n}{\GL_{n}\left(\finiteField\right)}}{W\left(g\right)W'\left(g\right)\mathcal{F}_{\fieldCharacter}\phi\left(e_1 \transpose{g}^{-1}\right)},$$
	where $e_1 = \left(1,0,\dots,0\right) \in \finiteField^n$, $e_n = \left(0,\dots,0,1\right) \in \finiteField^n$ and $\mathcal{F}_{\fieldCharacter}\phi$ is the Fourier transform of $\phi$ given by $$\mathcal{F}_{\fieldCharacter}\phi\left(x\right) = \sum_{y \in \finiteField^n}{\phi\left(y\right) \fieldCharacter\left(\standardForm{x}{y}\right)}.$$
\end{theorem}
\begin{remark}
	We normalize the Fourier transform like in \cite{Ye18}, so that the tensor product $\epsilon_0$-factors and the Rankin--Selberg gamma factors will satisfy the same relation for both $n<m$ and $n=m$. This relation is \Cref{thm:relation-between-epsilon-factors-and-gamma-factors} and it will be discussed in the next sections.
\end{remark}

We have the following expression for the Rankin--Selberg gamma factor in terms of Bessel functions for the case $n > m$.
\begin{proposition}[{\cite[Lemma 6.1.4]{Roditty10} or \cite[Proposition 2.16]{Nien14}}]\label{prop:rankin-selberg-gamma-factor-bessel-function-formula}Suppose $n > m$ and $\pi$ is cuspidal. Then
	\begin{equation}\label{eq:rankin-selberg-gamma-factor-bessel-function-formula}
		\gamma\left( \pi \times \sigma , \fieldCharacter \right) = \sum_{g \in \lquot{\UnipotentSubgroup_m}{\GL_m \left(\finiteField\right)}}{ \besselFunction_{\pi, \fieldCharacter}\begin{pmatrix}
			& \IdentityMatrix{n - m}\\
			g &
			\end{pmatrix} } \besselFunction_{\sigma, \fieldCharacter^{-1}}\left(g\right).
	\end{equation}
\end{proposition}

We also have an expression for the case $n=m$.
\begin{proposition}[{\cite[(16)]{Ye18}}]\label{prop:rankin-selberg-gamma-factor-bessel-function-formula-m-equals-n}Suppose $n=m$, and $\pi$ and $\sigma$ are cuspidal. Then
	\begin{equation}\label{eq:rankin-selberg-gamma-factor-bessel-function-formula-m-equals-n}
		\gamma\left(\pi \times \sigma, \fieldCharacter \right) = \sum_{g \in \lquot{\UnipotentSubgroup_n}{\GL_{n}\left(\finiteField\right)}}{\besselFunction_{\pi,\fieldCharacter}\left(g\right) \besselFunction_{\sigma, \fieldCharacter^{-1}}\left(g\right)\fieldCharacter\left(e_1 \transpose{g^{-1}} \transpose{e_n} \right)},
	\end{equation}
	where $e_1 = \left(1,0,\dots,0\right)$ and $e_n = \left(0,\dots,0,1\right)$.
\end{proposition}
In the special case where $\pi \cong \Contragradient{\sigma}$, we actually have by \cite[Corollary 4.3]{Ye18} that $$\gamma\left(\pi \times \Contragradient{\pi}, \fieldCharacter\right) = -1.$$

\subsection{Shahidi gamma factors}

Let $\pi$ and $\sigma$ be irreducible generic representations of $\GL_n\left(\finiteField\right)$ and $\GL_m\left(\finiteField\right)$, respectively.

In an unpublished note from 1979 \cite{soudry1979}, Soudry defines an intertwining operator $U \colon \sigma \circ \pi \rightarrow \pi \circ \sigma$, which allows him to define a gamma factor, which he denotes $\Gamma\left(\pi \times \sigma, \fieldCharacter\right)$. This is a finite field analog of the Langlands--Shahidi gamma factor \cite{shahidi1984fourier}. Soudry's results now appear in \cite{SoudryZelingher2023}.

The Shahidi gamma factor can be defined for any pair of irreducible generic representations $\pi$ and $\sigma$, regardless whether $n > m$ or whether $\pi$ or $\sigma$ are cuspidal. 

The Shahidi gamma factor satisfies \cite[Corollary 3.8]{SoudryZelingher2023} $$\Gamma\left(\pi \times \sigma, \fieldCharacter\right) = \centralCharacter{\pi}\left(-1\right)^m \centralCharacter{\sigma}\left(-1\right)^n
\Gamma\left(\Contragradient{\sigma} \times \Contragradient{\pi}, \fieldCharacter\right),$$
where $\centralCharacter{\pi}$ and $\centralCharacter{\sigma}$ are the central characters of $\pi$ and $\sigma$, respectively.

The Shahidi gamma factor $\Gamma\left({\pi \times \sigma}, \fieldCharacter\right)$ can be expressed in terms of the associated Bessel functions.

\begin{theorem}[{\cite[Theorem 3.14]{SoudryZelingher2023}}]\label{thm:soudry-gamma-factor-bessel-expression}
	\begin{enumerate}
		\item If $n > m$, then
		$$\Gamma\left({\pi \times \sigma}, \fieldCharacter\right) = q^{\frac{m}{2}\left(2 n - m - 1\right)} \centralCharacter{\sigma}\left(-1\right) \sum_{g \in \lquot{\UnipotentSubgroup_m}{\GL_m \left(\finiteField\right)}} \besselFunction_{\pi,\fieldCharacter}\begin{pmatrix}
		& \IdentityMatrix{n - m}\\
		g
		\end{pmatrix} \besselFunction_{\Contragradient{\sigma}, \fieldCharacter^{-1}}\left(g\right).$$
		\item  If $n = m$, then
		$$\Gamma\left({\pi \times \sigma}, \fieldCharacter\right) = q^{\frac{n \left(n-1\right)}{2}} \centralCharacter{\sigma}\left(-1\right) \sum_{g \in \lquot{\UnipotentSubgroup_n}{\GL_n \left(\finiteField\right)}} \fieldCharacter\begin{pmatrix}
		\IdentityMatrix{n} & g^{-1}\\
		& \IdentityMatrix{n}
		\end{pmatrix} \besselFunction_{\pi,\fieldCharacter}\left(g\right) \besselFunction_{\Contragradient{\sigma}, \fieldCharacter^{-1}}\left(g\right).$$
		\item If $n < m$, then
		\begin{equation*}
			\Gamma\left({\pi \times \sigma}, \fieldCharacter\right) = q^{\frac{n}{2}\left(2 m - n - 1\right)} \centralCharacter{\pi}\left(-1\right) \sum_{g \in \lquot{\UnipotentSubgroup_n}{\GL_n \left(\finiteField\right)}} \besselFunction_{\pi,\fieldCharacter}\left(g\right) \besselFunction_{\Contragradient{\sigma}, \fieldCharacter^{-1}}\begin{pmatrix}
			& \IdentityMatrix{m - n}\\
			g
			\end{pmatrix}.
		\end{equation*}
	\end{enumerate}
\end{theorem}
The Shahidi gamma factor is multiplicative.
\begin{theorem}[{Multiplicativity of gamma factors, \cite[Theorem 3.9]{SoudryZelingher2023}}]\label{thm:soudry-gamma-factor-multiplicative}
	Let $\pi$ and $\sigma$ be irreducible generic representations of $\GL_{n}\left(\finiteField\right)$ and $\GL_{m}\left(\finiteField\right)$, respectively. Suppose that $\sigma_1$ and $\sigma_2$ are irreducible generic representations of $\GL_{m_1}\left(\finiteField\right)$ and $\GL_{m_2}\left(\finiteField\right)$, respectively, such that $m_1 + m_2 = m$, and suppose that $\sigma \subseteq \sigma_1 \circ \sigma_2$. Then
	$$ \Gamma\left({\pi \times \sigma}, \fieldCharacter\right) = \Gamma\left({\pi \times \sigma_1}, \fieldCharacter\right) \Gamma\left({\pi \times \sigma_2}, \fieldCharacter\right).$$
\end{theorem}

We notice that if $\pi$ is cuspidal and $n > m$, then by \Cref{prop:rankin-selberg-gamma-factor-bessel-function-formula} and \Cref{thm:soudry-gamma-factor-bessel-expression}, $$\gamma\left(\pi \times \sigma, \fieldCharacter \right) = \centralCharacter{\sigma}\left(-1\right) q^{-\frac{m}{2}\left(2n -m -1\right)} \Gamma\left({\pi \times \Contragradient{\sigma}}, \fieldCharacter\right).$$
This formula also holds for $n = m$ if $\pi$ and $\sigma$ are cuspidal by  \Cref{prop:rankin-selberg-gamma-factor-bessel-function-formula-m-equals-n} and \Cref{thm:soudry-gamma-factor-bessel-expression}.

We extend our definition of $\gamma\left(\pi \times \sigma,\fieldCharacter\right)$ to all irreducible generic representations, by defining $\gamma\left(\pi \times \sigma,\fieldCharacter\right) = \centralCharacter{\sigma}\left(-1\right) q^{-\frac{m}{2}\left(2n -m -1\right)} \Gamma\left({\pi \times \Contragradient{\sigma}}, \fieldCharacter\right)$. Then under this notation we have the following theorem.

\begin{theorem}\label{thm:shahidi-gamma-factor-properties}Let $\pi$ and $\sigma$ be irreducible generic representations of $\GL_n\left(\finiteField\right)$ and $\GL_{m}\left(\finiteField\right)$, respectively. Let $\centralCharacter{\pi}$ and $\centralCharacter{\sigma}$ denote the central characters of $\pi$ and $\sigma$, respectively.
	\begin{enumerate}
		\item If $n > m$, then $\gamma\left(\pi \times \sigma, \fieldCharacter\right)$ is given by \eqref{eq:rankin-selberg-gamma-factor-bessel-function-formula}.
		\item If $n = m$, then $\gamma\left(\pi \times \sigma, \fieldCharacter\right)$ is given by \eqref{eq:rankin-selberg-gamma-factor-bessel-function-formula-m-equals-n}.
		\item If $n < m$, then
		\begin{equation}
			\label{eq:m-greater-than-n-rankin-selberg} \gamma\left({\pi \times \sigma, \fieldCharacter}\right) = q^{ \frac{m\left(m+1\right)}{2} - \frac{n\left(n+1\right)}{2} } 
			\sum_{g \in \lquot{\UnipotentSubgroup_n}{\GL_n \left(\finiteField\right)}} \besselFunction_{\pi,\fieldCharacter}\left(g\right) \besselFunction_{\sigma, \fieldCharacter^{-1}}\begin{pmatrix}
			& -\IdentityMatrix{m - n}\\
			g
			\end{pmatrix}.
		\end{equation}		
		\item For any $n$ and $m$, \label{item:gamma-factor-swap-representations-formula} $$\gamma\left(\pi \times \sigma, \fieldCharacter \right) = q^{\frac{m\left(m+1\right)}{2} - \frac{n\left(n+1\right)}{2}} \centralCharacter{\pi}\left(-1\right)^{m-1} \centralCharacter{\sigma}\left(-1\right)^{n-1} \gamma\left(\sigma \times \pi, \fieldCharacter \right).$$

		\item For any $n$ and $m$, suppose that $m = m_1 + m_2$, and let $\sigma_1$ and $\sigma_2$ be irreducible generic representations of $\GL_{m_1}\left(\finiteField\right)$ and $\GL_{m_2}\left(\finiteField\right)$, respectively. Suppose that $\sigma$ is the unique irreducible generic subrepresentation of the parabolic induction $\sigma_1 \circ \sigma_2$. Then 
		$$ \gamma\left(\pi \times \sigma, \fieldCharacter\right) = q^{m_1 m_2} \gamma\left(\pi \times \sigma_1, \fieldCharacter\right) \gamma\left(\pi \times \sigma_2, \fieldCharacter\right).$$
		\item For any $n$ and $m$, suppose that $n = n_1 + n_2$, and let $\pi_1$ and $\pi_2$ be irreducible generic representations of $\GL_{n_1}\left(\finiteField\right)$ and $\GL_{n_2}\left(\finiteField\right)$, respectively. Suppose that $\pi$ is the unique irreducible generic subrepresentation of the parabolic induction $\pi_1 \circ \pi_2$. Then
		$$ \gamma\left(\pi \times \sigma, \fieldCharacter\right) = q^{-\frac{m\left(m+1\right)}{2}} \centralCharacter{\sigma}\left(-1\right) \gamma\left(\pi_1 \times \sigma, \fieldCharacter\right) \gamma\left(\pi_2 \times \sigma, \fieldCharacter\right).$$		
	\end{enumerate}
\end{theorem}

\subsection{A recursive expression for the Bessel function}

Using \Cref{thm:bessel-function-as-sum-of-trace} and the orthogonality relations of trace characters of irreducible representations, one can also show the following proposition.
\begin{proposition}[{\cite[Lemma 4.2]{Ye18}}]\label{prop:orthogonality-for-m-equals-n}
	Let $\representationDeclaration{\sigma}$ and $\representationDeclaration{\sigma'}$ be two irreducible generic representations of $\GL_{m}\left(\finiteField\right)$. Then
	\begin{equation}\label{eq:orthogonality-of-bessel-functions}
	\sum_{g \in \lquot{\UnipotentSubgroup_m}{\GL_m\left(\finiteField\right)}} \besselFunction_{\sigma, \fieldCharacter} \left(g\right) \besselFunction_{\sigma', \fieldCharacter^{-1}} \left(g\right) = \begin{dcases}
	\frac{\grpIndex{\GL_m\left(\finiteField\right)}{\UnipotentSubgroup_m}}{\dim \sigma} & \Contragradient{\sigma} \cong \sigma'\\
	0 & \text{otherwise}
	\end{dcases}.
	\end{equation}
\end{proposition}

We will use \Cref{prop:orthogonality-for-m-equals-n} and \Cref{thm:shahidi-gamma-factor-properties} to prove the following theorem.
\begin{theorem}\label{thm:recursive-bessel-using-gamma-factors}
	Let $\pi$ be an irreducible generic representation of $\GL_{n}\left(\finiteField\right)$, and let $g \in \GL_m \left(\finiteField\right)$.
	Denote $$\mathcal{F}_{\pi, m, \fieldCharacter}\left(g\right) = \frac{1}{\grpIndex{\GL_m\left(\finiteField\right)}{\UnipotentSubgroup_m}}\sum_{\sigma} {{\dim \sigma} \cdot \gamma\left(\pi \times \Contragradient{\sigma}, \fieldCharacter\right)}  \cdot \besselFunction_{\sigma, \fieldCharacter}\left(g\right),$$
	where $\sigma$ runs over all the irreducible generic representations of $\GL_m \left(\finiteField\right)$. Then	
	\begin{enumerate}
		\item If $m<n$,
			$$\mathcal{F}_{\pi, m, \fieldCharacter}\left(g\right) = \besselFunction_{\pi, \fieldCharacter} \begin{pmatrix}
		0 & \IdentityMatrix{n-m}\\
		g & 0
		\end{pmatrix}.$$
		\item If $m=n$,
		$$\mathcal{F}_{\pi, m, \fieldCharacter}\left(g\right) = \besselFunction_{\pi, \fieldCharacter}\left(g\right) \fieldCharacter \begin{pmatrix}
		I_n & g^{-1}\\
		& I_n
		\end{pmatrix}.$$
		\item If $m > n$, then $\mathcal{F}_{\pi, m, \fieldCharacter}\left(g\right)=0$ unless $g = u_1 \left(\begin{smallmatrix}
		& -\IdentityMatrix{m-n}\\
		h
		\end{smallmatrix}\right) u_2$, for $u_1, u_2 \in \UnipotentSubgroup_m$, and $h \in \GL_n\left(\finiteField\right)$, in which case $$\mathcal{F}_{\pi, m, \fieldCharacter}\left(g\right) = q^{m^2 - n^2 - \binom{m-n}{2}} \fieldCharacter\left(u_1 u_2\right) \besselFunction_{\pi, \fieldCharacter}\left(h\right).$$
		In particular, for $c \in \multiplicativegroup{\finiteField}$ we have $\mathcal{F}_{\pi, m, \fieldCharacter}\left(c \IdentityMatrix{m} \right)=0$.

	\end{enumerate}

\end{theorem}
\begin{proof}
	Denote by $\mathcal{G}_m$ the set of irreducible generic representations of $\GL_{m}\left(\finiteField\right)$, and denote by $\mathfrak{g}_{m}$ the set of matrices of the form $g_{m_1,\dots,m_s}\left(c_1,\dots,c_s\right)$ with $m_1 + \dots + m_s = m$, where $m_1,\dots,m_s > 0$ and $c_1,\dots,c_s \in \multiplicativegroup{\finiteField}$.
	
	For simplicity, we first assume $m < n$. We restate \Cref{prop:orthogonality-for-m-equals-n} and \eqref{eq:rankin-selberg-gamma-factor-bessel-function-formula} in matrix form.  To begin, notice that the summand of the left hand side of \eqref{eq:orthogonality-of-bessel-functions} is invariant under right multiplication by elements of $\UnipotentSubgroup_m$. Therefore using \Cref{prop:support-of-bessel-function}, we can rewrite \Cref{prop:orthogonality-for-m-equals-n} in the following form. \begin{equation}\label{eq:bessel-orhtogonality-matrix-form}
		\sum_{h \in \mathfrak{g}_m}{t_{h} \cdot \besselFunction_{\sigma, \fieldCharacter}\left(h\right)\besselFunction_{\sigma', \fieldCharacter^{-1}}\left(h\right)} = \begin{dcases}
		\frac{\grpIndex{\GL_m\left(\finiteField\right)}{\UnipotentSubgroup_m}}{\dim \sigma} & \Contragradient{\sigma} \cong \sigma'\\
		0 & \text{otherwise}
		\end{dcases},
	\end{equation}
	where for $h \in \mathfrak{g}_m$, $t_{h}$ is the cardinality of the set $\left\{ \UnipotentSubgroup_m h u \mid u \in \UnipotentSubgroup_m \right\}$. One can show that if $h = g_{m_1,\dots,m_s}\left(c_1,\dots,c_s\right)$, then $t_{h} = q^{\binom{m}{2} - \sum_{j=1}^{r}{\binom{m_j}{2}}}$, see for instance \cite[Lemma 2.29]{YeZeligher18}, but we will not need this. Similarly, we have that the summand in \eqref{eq:rankin-selberg-gamma-factor-bessel-function-formula} is invariant under right multiplication by elements of $\UnipotentSubgroup_m$, and hence we have by \Cref{thm:shahidi-gamma-factor-properties} that
	\begin{equation}\label{eq:bessel-rankin-selberg-matrix-form}
		\gamma\left( \pi \times \sigma , \fieldCharacter \right) = \sum_{h \in \mathfrak{g}_m}{t_h \cdot  \besselFunction_{\pi, \fieldCharacter}\begin{pmatrix}
			& \IdentityMatrix{n - m}\\
			h
			\end{pmatrix} \besselFunction_{\sigma, \fieldCharacter^{-1}}\left(h\right)}.
	\end{equation}
	
	Let $B_{m,\fieldCharacter}$ be a matrix whose rows and columns are indexed by $\mathcal{G}_m$ and $\mathfrak{g}_m$, respectively, having the value $\besselFunction_{\sigma, \fieldCharacter}\left(h\right)$ at position $\left(\sigma, h\right)$. Then $B_{m, \fieldCharacter}$ is a square matrix of size $\left(q^m - q^{m-1}\right)\times\left(q^m - q^{m-1}\right)$. By \Cref{prop:complex-conjugate-of-bessel-function}, we have that its conjugate transpose $B_{m,\fieldCharacter}^\ast$ is a matrix whose rows are indexed by $\mathfrak{g}_m$, whose columns indexed by $\mathcal{G}_m$, having the value $\besselFunction_{\Contragradient{\sigma}, \fieldCharacter^{-1}}\left(h\right)$ at position $ \left(h,\sigma\right)$.
	
	Let $T_{m}$ be a diagonal square matrix whose rows and columns are indexed by $\mathfrak{g}_m$, having the value $t_{h}$ as the $h \in \mathfrak{g}_m$ entry of its diagonal. Let $D_{m}$ be a diagonal square matrix whose rows and columns are indexed by $\mathcal{G}_m$, having the value $\frac{\dim \sigma}{\grpIndex{\GL_{m}\left(\finiteField\right)}{\UnipotentSubgroup_m}}$ as the $\sigma \in \mathcal{G}_m$ entry of its diagonal. Then by \eqref{eq:bessel-orhtogonality-matrix-form}, we have the relation $D_m B_{m,\fieldCharacter} T_m B_{m,\fieldCharacter}^{\ast} = I_{\mathcal{G}_m}$, where $I_{\mathcal{G}_m}$ is the identity matrix whose rows and columns are indexed by $\mathcal{G}_m$. Therefore $D_m B_{m,\fieldCharacter} T_m$ is the inverse of $B_{m,\fieldCharacter}^{\ast}$, and hence we have that $B_{m,\fieldCharacter}^{\ast} D_m B_{m,\fieldCharacter} T_m = I_{\mathfrak{g}_m}$, where the right hand side is the identity matrix whose rows and columns are indexed by $\mathfrak{g}_m$.
	
	Let $v_{\pi,m,\fieldCharacter}$ be the column vector whose rows are indexed by $\mathcal{G}_m$, having at position $\sigma$ the value $\gamma \left(\pi \times \sigma, \fieldCharacter\right)$, and let $b_{\pi,m,\fieldCharacter}$ the column vector indexed by $\mathfrak{g}_m$, having at position $h \in \mathfrak{g}_m$ the value $\besselFunction_{\pi, \fieldCharacter}\left(\begin{smallmatrix}
	& \IdentityMatrix{n - m}\\
	h
	\end{smallmatrix}\right)$. Then we have by \eqref{eq:bessel-rankin-selberg-matrix-form} that $B_{m,\fieldCharacter^{-1}} T_m b_{\pi,m,\fieldCharacter} = v_{\pi,m,\fieldCharacter}$. Multiplying both sides by $B_{m,\fieldCharacter^{-1}}^{\ast} D_m$ from the left, we get
	$$b_{\pi,m,\fieldCharacter} = B_{m,\fieldCharacter^{-1}}^{\ast} D_m v_{\pi,m,\fieldCharacter},$$
	which implies that $$\besselFunction_{\pi, \fieldCharacter} \begin{pmatrix}
	& \IdentityMatrix{n-m}\\
	h
	\end{pmatrix} = \sum_{\sigma \in \mathcal{G}_m}{\besselFunction_{\Contragradient{\sigma}, \fieldCharacter}\left(h\right) \cdot\frac{\dim \sigma}{\grpIndex{\GL_{m}\left(\finiteField\right)}{\UnipotentSubgroup_m}}} \cdot \gamma\left(\pi \times \sigma, \fieldCharacter\right),$$
	For any $h \in \mathfrak{g}_m$. Replacing $\sigma$ with $\Contragradient{\sigma}$ and using the equivariance properties of the Bessel function, we get the desired result.
	
	The case $m = n$ is treated similarly, by expressing the formula in \eqref{eq:rankin-selberg-gamma-factor-bessel-function-formula-m-equals-n} as a linear system $B_{m,\fieldCharacter^{-1}} T_m b_{\pi,m,\fieldCharacter} = v_{\pi,m,\fieldCharacter}$, where this time $b_{\pi, m, \fieldCharacter}$ is the column vector whose rows are indexed by $\mathfrak{g}_m$, having at position $h \in \mathfrak{g}_m$ the value $\besselFunction_{\pi, \fieldCharacter} \left(h\right) \fieldCharacter\left(\begin{smallmatrix}
	I_n & h^{-1}\\
	& I_n
	\end{smallmatrix}\right)$.
	
	In the case $m > n$ we write \eqref{eq:m-greater-than-n-rankin-selberg} as a linear system. This time $b_{\pi, m, \fieldCharacter}$ has zeros at entries that are not of the form $\left( \begin{smallmatrix}
	& -\IdentityMatrix{m - n}\\
	h
	\end{smallmatrix} \right)$, where $h \in \mathfrak{g}_n$, and has the value $\besselFunction_{\pi, \fieldCharacter}\left( h\right)$ at position $\left( \begin{smallmatrix}
	& -\IdentityMatrix{m - n}\\
	h
	\end{smallmatrix} \right)$ for $h \in \mathfrak{g}_n$. We use the fact that $$t_{\left( \begin{smallmatrix}
		& -\IdentityMatrix{m - n}\\
		h
		\end{smallmatrix} \right)} = t_{h} \cdot q^{ \binom{m}{2} - \binom{n}{2} - \binom{m-n}{2} } .$$
	\end{proof}

\begin{remark}
	In her master's thesis, Roditty-Gershon gave a different proof for a similar formula \cite[Theorem 6.2.1, Equation (6.23)]{Roditty10}. She proved that if $\pi$ is an irreducible cuspidal representation of $\GL_n\left(\finiteField\right)$, then for any $m < n$, and any $g \in \GL_m\left(\finiteField\right)$, the following formula holds:
	$$ \besselFunction_{\pi, \fieldCharacter}\begin{pmatrix}
	& I_{n - m}\\
	g
	\end{pmatrix} =  \sum_{\sigma} \abs{t_\sigma}^2 \gamma\left(\pi \times \sigma^{\vee}, \fieldCharacter\right) \besselFunction_{\sigma, \fieldCharacter}\left(g\right),$$
	where $\sigma$ goes over all irreducible generic representations of $\GL_m\left(\finiteField\right)$, and $$\abs{t_\sigma}^2 = \frac{1}{\sum_{h \in \UnipotentSubgroup_m \backslash \GL_m\left(\finiteField\right)} \abs{\besselFunction_{\sigma, \fieldCharacter}\left(h\right)}^2 }.$$
\end{remark}

\subsection{Epsilon factors}

In \cite{ye2021epsilon}, Rongqing Ye and the author defined $\epsilon_0$-factors, and in particular tensor product $\epsilon_0$-factors. We briefly review the results we need from there.

For two irreducible cuspidal representations, their tensor product $\epsilon_0$-factor can be expressed as a product of Gauss sums of their corresponding regular multiplicative characters.
\begin{theorem}[{\cite[Theorem 4.2]{ye2021epsilon}}]\label{thm:epsilon-factor-of-tensor-product-as-product-of-gauss-sums}
	Let $f$ and $g$ be Frobenius orbits of degrees $n$ and $m$, respectively, and let $\Pi_f$ and $\Pi_g$ be their corresponding irreducible cuspidal representations of $\GL_{n}\left(\finiteField\right)$ and $\GL_{m}\left(\finiteField\right)$, respectively. Let $\alpha \in \charactergroup{n}$ and $\beta \in \charactergroup{m}$ be characters such that $\alpha$ and $\beta$ represent $f$ and $g$, respectively. Then
	$$\epsilon_0\left(\Pi_f \times \Pi_g, \fieldCharacter\right) = \left(-1\right)^{nm} q^{-\frac{nm}{2}} \prod_{k = 1}^{\gcd\left(n,m\right)}{ \tau\left( \alpha \circ \FieldNorm{\lcm\left(n,m\right)}{n} \cdot \beta^{q^{k-1}}\circ\FieldNorm{\lcm\left(n,m\right)}{m},\fieldCharacter_{\lcm\left(n,m\right)}\right) },$$
	where for every $r$ and every $\gamma \in \charactergroup{r}$, $\tau\left(\gamma, \fieldCharacter_r\right)$ is the Gauss sum $$\tau\left(\gamma, \fieldCharacter_r\right) = -\sum_{\xi \in \multiplicativegroup{\finiteFieldExtension{r}}}{\gamma^{-1}\left(\xi\right) \fieldCharacter_r \left(\xi\right)},$$
	where $\fieldCharacter_r = \fieldCharacter \circ \FieldTrace_{\FieldExtension{\finiteFieldExtension{r}}{\finiteField}}$.
\end{theorem}

In order to compute the tensor product $\epsilon_0$-factors for general irreducible representations, we use the following multiplicativity property \cite[Theorem 4.1]{ye2021epsilon}.
\begin{theorem}\label{thm:epsilon-factors-are-multiplicative}
	Let $\phi \in P_n\left(\limitcharactergroup\right)$ and $\phi' \in P_m\left(\limitcharactergroup\right)$ parameterize irreducible representations of $\GL_{n}\left(\finiteField\right)$ and $\GL_m\left(\finiteField\right)$, respectively. Then $$ \epsilon_0\left(\Pi_\phi \times \Pi_{\phi'}, \fieldCharacter\right) = \prod_{f,g \in \lquot{\Frobenius}{\limitcharactergroup}}{\epsilon_0\left(\Pi_f \times \Pi_g, \fieldCharacter\right)}^{\sizeof{\phi\left(f\right)} \cdot \sizeof{\phi'\left(g\right)}},$$
	where $\lquot{\Frobenius}{\limitcharactergroup}$ is the set of all the Frobenius orbits of $\limitcharactergroup$.
\end{theorem}

Next, we give a relation between the Rankin--Selberg gamma factors and the tensor product $\epsilon_0$-factors.
\begin{theorem}\label{thm:relation-between-epsilon-factors-and-gamma-factors-cuspidal}Suppose $n\ge m$. Let $\pi$ and $\sigma$ be irreducible cuspidal representations of $\GL_{n}\left(\finiteField\right)$ and $\GL_{m}\left(\finiteField\right)$, respectively. Denote by $\centralCharacter{\sigma}$ the central character of $\sigma$. Then
	$$\gamma\left(\pi \times \sigma, \fieldCharacter\right) = q^{-\frac{m{\left(n -m -1\right)}}{2}} \centralCharacter{\sigma}\left(-1\right)^{n-1}\epsilon_0\left(\pi \times \sigma, \fieldCharacter\right).$$
\end{theorem}
\begin{proof}
	This was done for $n > m$ in \cite[Theorem 4.4]{ye2021epsilon}. We are left to consider the case $n = m$. For $n = m$, and $\pi \not\cong \Contragradient{\sigma}$, we have an equality of the Rankin--Selberg gamma factor corresponding to $\pi \times \sigma$ and the Rankin--Selberg gamma factor corresponding to $\Pi \times \Sigma$, where $\Pi$ and $\Sigma$ are level zero representations constructed from $\pi$ and $\sigma$ respectively \cite[Theorem 4.1]{Ye18}. Now one can proceed exactly as in the proof of \cite[Theorem 4.4]{ye2021epsilon}. The normalization of the Fourier transform $\mathcal{F}_{\fieldCharacter}$ contributes the constant $q^{-\frac{n\left(n - n - 1\right)}{2}} = q^{\frac{n}{2}}$.
	
	We are left to deal with the case $m = n$ and $\pi \cong \Contragradient{\sigma}$. In this case (under our normalization), $\gamma\left(\pi \times \Contragradient{\pi}, \fieldCharacter \right) = -1$, see \cite[Corollary 4.3]{Ye18} or \cite[Theorem A.1]{SoudryZelingher2023}. We therefore need to show that $$\epsilon_0\left(\pi \times \Contragradient{\pi}, \fieldCharacter\right) = -\centralCharacter{\pi}\left(-1\right)^{n-1} q^{-\frac{n}{2}},$$ where $\centralCharacter{\pi}$ is the central character of $\pi$. Let $\pi = \Pi_f$, where $f$ is a Frobenius orbit of degree $n$, and let $\alpha \in \charactergroup{n}$ be a representative of $f$. Then by \Cref{thm:epsilon-factor-of-tensor-product-as-product-of-gauss-sums} we have
	$$\epsilon_0\left(\pi \times \Contragradient{\pi}, \fieldCharacter\right) = \left(-1\right)^{n^2} q^{-\frac{n^2}{2}} \prod_{i=0}^{n-1}{\tau\left(\alpha \cdot \alpha^{-q^{i}} ,\fieldCharacter_n\right)}.$$
	Notice that $$\alpha^{1-q^i} = \left(\left(\alpha^{1-q^{n-i}}\right)^{{-1}}\right)^{q^{i}}.$$
	Since $\tau\left(\beta^{-1}, \fieldCharacter_n\right) = \conjugate{\tau\left(\beta, \fieldCharacter_n\right)} \beta\left(-1\right)$ for $\beta \in \charactergroup{n}$, and since Gauss sums are constant for multiplicative characters in the same Frobenius orbit, we have that $$\tau \left(\alpha^{1 - q^i}, \fieldCharacter_n\right) \tau \left(\alpha^{1 - q^{n-i}}, \fieldCharacter_n\right) = \abs{\tau\left(\alpha^{1 - q^i}, \fieldCharacter_n\right)}^2 \alpha^{1 - q^i}\left(-1\right) = q^n.$$
	For odd $n=2n'+1$, we therefore get $$\epsilon_0\left(\pi \times \Contragradient{\pi}, \fieldCharacter\right) = -q^{-\frac{n^2}{2}} \cdot \tau\left(1, \fieldCharacter_n\right) \cdot \left(q^{n}\right)^{\frac{n - 1}{2}} = -q^{-\frac{n}{2}} = -q^{-\frac{n}{2}}\centralCharacter{\pi}\left(-1\right)^{n-1}.$$
	For even $n=2n'$, we get $$\epsilon_0\left(\pi \times \Contragradient{\pi}, \fieldCharacter\right) = q^{-\frac{n^2}{2}} \cdot \tau\left(1, \fieldCharacter_n \right) \cdot \tau\left(\alpha^{1-q^{n'}},\fieldCharacter_n\right) \cdot \left(q^n\right)^{\frac{n-2}{2}} = q^{-n} \tau\left(\alpha^{1-q^{n'}}, \fieldCharacter_n \right).$$ Notice that $\alpha^{1-q^{n'}}$ is trivial on $\multiplicativegroup{\finiteFieldExtension{n'}}$. By \Cref{prop:gauss-sums-even-field-extension-lemma}, $\tau\left(\alpha^{1 - q^{n'}}, \fieldCharacter_{2n'} \right) = -q^{n'} \cdot \alpha^{q^{n'} - 1} \left(z\right)$, where $z \in \multiplicativegroup{\finiteFieldExtension{2n'}}$ satisfies $z^{q^{n'}-1}=-1$. Then $\epsilon_0\left(\pi \times \Contragradient{\pi}, \fieldCharacter\right) =  -q^{-\frac{n}{2}} \alpha\left(-1\right)$, and the statement follows since $\alpha\left(-1\right) = \centralCharacter{\pi}\left(-1\right) = \centralCharacter{\pi}\left(-1\right)^{n-1}$.
\end{proof}

In order to relate $\gamma\left(\pi \times \sigma, \fieldCharacter\right)$ to $\epsilon_0\left(\pi \times \sigma, \fieldCharacter\right)$ in the general case, we use  \Cref{thm:relation-between-epsilon-factors-and-gamma-factors-cuspidal} combined with the multiplicativity property of the tensor product $\epsilon_0$-factors (\Cref{thm:epsilon-factors-are-multiplicative}) and the properties of $\gamma\left(\pi \times \sigma,\fieldCharacter\right)$ shown by Soudry (\Cref{thm:shahidi-gamma-factor-properties}). Using these repeatedly, we get the following result.
\begin{theorem}\label{thm:relation-between-epsilon-factors-and-gamma-factors}
	Let $\pi$ and $\sigma$ be irreducible generic representations of $\GL_{n}\left(\finiteField\right)$ and $\GL_{m}\left(\finiteField\right)$, respectively. Then
	$$ \gamma\left(\pi\times\sigma,\fieldCharacter\right) = q^{-\frac{m \left(n - m - 1\right)}{2}} \centralCharacter{\sigma}\left(-1\right)^{n-1} \epsilon_0\left(\pi \times \sigma, \fieldCharacter\right).$$
\end{theorem}

\section{Étale algebras and exotic Kloosterman sums and sheaves}

In this section, we review the definition of an étale algebra over $\finiteField$. We specialize to the tensor product algebra $\finiteFieldExtension{n} \otimes_{\finiteField} \finiteFieldExtension{m}$, and show that it can be used to rewrite the Gauss sums we encountered in the previous section in a more compact way. Finally, we give a brief review of Katz's exotic Kloosterman sums and sheaves. We will use these later and relate them to special values of the Bessel function.

\subsection{Étale algebras}

Let $B$ be a finite-dimensional $\finiteField$-algebra. We say that $B$ is an étale algebra of rank $n$ if $B \cong \prod_{i = 1}^r{ \finiteFieldExtension{n_i}}$, where $n_1,\dots,n_r$ are positive integers with $n_1 + \dots + n_r = n$. Let $\multiplicativegroup{B}$ be the multiplicative group of $B$. We have that $\multiplicativegroup{B} \cong \prod_{i = 1}^r{ \multiplicativegroup{\finiteFieldExtension{n_i}}}$. We denote by $\widehat{\multiplicativegroup{B}}$ the group of complex multiplicative characters $\chi \colon \multiplicativegroup{B} \rightarrow \multiplicativegroup{\cComplex}$. Then $\widehat{\multiplicativegroup{B}} \cong \prod_{i=1}^r \charactergroup{n_i}$.

Let $R$ be an $\finiteField$-algebra. We consider the tensor product $B \otimes_\finiteField R$. We have the norm map, $\etaleNorm^{B \otimes_\finiteField R}_2 \colon B \otimes_\finiteField R \rightarrow R$, defined as follows. Let $x \in B \otimes_\finiteField R$. We consider the multiplication map $T_x \colon B \otimes_\finiteField R \rightarrow B \otimes_\finiteField R$, defined by $T_x\left(y\right) = xy$. The norm $\etaleNorm^{B \otimes_\finiteField R}_2 \left(x\right)$ is defined as the determinant of $T_x$, viewed as a $R$-linear map. Similarly, we have the trace map $\trace^{B \otimes_\finiteField R}_2 \colon B \otimes_\finiteField R \rightarrow R$ given by taking the trace of $T_x$, viewed as a $R$-linear map.

If $R$ is a finite-dimensional $\finiteField$-algebra, then we also have the corresponding norm $\etaleNorm^{B \otimes_\finiteField R}_1 \colon B \otimes_\finiteField R \rightarrow B$ and trace $\trace^{B \otimes_\finiteField R}_1 \colon B \otimes_\finiteField R \rightarrow B$ maps to $B$, defined by taking the determinant and trace of $T_x$, respectively, this time viewed as a $B$-linear map. In this case, we also denote $\trace x = \trace T_x$, where we view $T_x$ as an $\finiteField$-linear map.

\subsection{Tensor product of finite fields} \label{sec:tensor-product-of-finite-fields}
Let $n$ and $m$ be positive integers. Denote $d = \gcd \left(n,m\right)$ and $l = \lcm \left(n,m\right)$. We consider the tensor product $\finiteFieldExtension{n} \otimes_\finiteField \finiteFieldExtension{m}$.

We have that $\finiteFieldExtension{n} \otimes_\finiteField \finiteFieldExtension{m} \cong \finiteFieldExtension{l}^d$, by the following isomorphism. Write $\finiteFieldExtension{m} = \finiteField\left[ \theta_m \right] = \finiteField \left[ X \right] \slash \left( p_m \left(X\right) \right)$, where $p_m\left(X\right) \in \finiteField\left[X\right]$ is an irreducible polynomial of degree $m$, and $\theta_m \in \finiteFieldExtension{m}$ is a root of $p_m\left(X\right)$. Then $p_m\left(X\right) = \prod_{j=1}^m (X - \theta_m^{1/q^{j-1}})$. We have $\finiteFieldExtension{n} \otimes_\finiteField \finiteFieldExtension{m} \cong \finiteFieldExtension{n}\left[X\right] / \left( p_m\left(X\right) \right)$ and the last ring is isomorphic to $\finiteFieldExtension{l}^d$ by mapping $P\left(X\right) \in \finiteFieldExtension{n}\left[X\right] / \left( p_m\left(X\right) \right)$ to $(P(\theta_m),P(\theta_m^{1/q}),\dots, P(\theta_m^{1/q^{d-1}}))$.

Under this isomorphism, $s_n \in \finiteFieldExtension{n}$ acts on $\left(x_1, \dots, x_d\right) \in \finiteFieldExtension{l}^d$ by $$\left(s_n \otimes 1\right) \left(x_1, \dots, x_d\right) = \left(s_n x_1, \dots, s_n x_d\right),$$ while $s_m \in \finiteFieldExtension{m}$ acts by $$\left(1 \otimes s_m\right) \left(x_1, \dots, x_d \right) = ( s_m x_1,  s_m^{1/q} x_2, \dots, s_m^{1/q^{d-1}} x_d ).$$

We denote the norm maps defined in the previous section by $\etaleNorm_{1}^{n,m} = \etaleNorm_1^{\finiteFieldExtension{n} \otimes_\finiteField \finiteFieldExtension{m}} \colon \finiteFieldExtension{n} \otimes_\finiteField \finiteFieldExtension{m} \rightarrow \finiteFieldExtension{n} $ and $\etaleNorm_{2}^{n,m} = \etaleNorm_2^{\finiteFieldExtension{n} \otimes_\finiteField \finiteFieldExtension{m}} \colon \finiteFieldExtension{n} \otimes_\finiteField \finiteFieldExtension{m} \rightarrow \finiteFieldExtension{m} $.

Under the above isomorphism, we have that the norms $\etaleNorm_1^{n,m}$ and $\etaleNorm_2^{n,m}$ of $\left( x_1,\dots,x_d \right) \in \finiteFieldExtension{l}^d$ are given by \begin{align*}
\etaleNorm_1^{n,m}\left( x_1,\dots,x_d \right) &= \prod_{j=1}^d \FieldNorm{l}{n}\left( x_j\right),\\	
\etaleNorm_2^{n,m}\left( x_1,\dots,x_d \right) &= \prod_{j=1}^d \FieldNorm{l}{m}\left( x_j\right)^{q^{j-1}}.
\end{align*}

Indeed, the formula regarding $\etaleNorm_{1}^{n,m}$ is obvious. In order to see the formula regarding $\etaleNorm_{2}^{n,m}$, choose a basis $\mathcal{B} = (b_1,\dots,b_{l/m})$ of $\finiteFieldExtension{l}$ with respect to the base field $\finiteFieldExtension{m}$, and let $M_{x_i} \in \mathrm{Mat}_{l/m}\left(\finiteField_d\right)$ be the matrix representing the multiplication map $T_{x_i} \colon \finiteFieldExtension{l} \rightarrow \finiteFieldExtension{l}$ with respect to the basis $\mathcal{B}$. Let $\left(e_1,\dots,e_d\right)$ be the standard basis of $\finiteFieldExtension{l}^d$. Then, with respect to the following $\finiteFieldExtension{m}$-basis of $\finiteField_l^d$, $$\left(\mathcal{B} e_1, \mathcal{B} e_2, \dots, \mathcal{B}  e_d \right),$$ the multiplication map $T_{\left(x_1,\dots,x_d\right)}$ is represented by the matrix $$\diag\left(M_{x_1}, M_{x_2}^q, \dots, M_{x_d}^{q^{d-1}}\right),$$ and its determinant is $$\prod_{j=1}^d \FieldNorm{l}{m}\left(x_j\right)^{q^{j-1}}.$$

Similarly, under the above isomorphism, we have that the trace of $\left(x_1,\dots,x_d\right) \in \finiteFieldExtension{l}^d$ to $\finiteField$ is given by
$$ \trace \left(x_1,\dots,x_d\right) = \sum_{j=1}^d { \FieldTrace_{\finiteFieldExtension{l} \slash \finiteField}\left( x_j \right) }.$$

This formalism allows us to rewrite the epsilon factors discussed previously in a cleaner way. Let $\alpha \in \charactergroup{n}$ and $\beta \in \charactergroup{m}$. Then
\begin{equation}\label{eq:equality-of-etale-algebra-gauss-sums}
\prod_{k = 1}^{d}{\tau \left(\alpha \circ \FieldNorm{l}{n} \cdot \beta^{q^{k-1}} \circ \FieldNorm{l}{m}, \fieldCharacter_{l} \right)} = \left(-1\right)^d \sum_{\xi \in \multiplicativegroup{\left(\finiteFieldExtension{n} \otimes_\finiteField \finiteFieldExtension{m}\right)} } \alpha^{-1}\left( \etaleNorm_1^{n,m} \left(\xi\right) \right) \beta^{-1} \left( \etaleNorm_2^{n,m}\left( \xi \right) \right) \fieldCharacter \left( \trace \xi \right).
\end{equation}
We denote the value in \eqref{eq:equality-of-etale-algebra-gauss-sums} by $\tau_{n,m}\left(\alpha \times \beta, \fieldCharacter\right)$. Since $d = \gcd\left(n,m\right)$ and $n+m+nm$ have the same parity, we have
\begin{equation}
\tau_{n,m}\left(\alpha \times \beta,\fieldCharacter\right) = \left(-1\right)^{nm + n + m} \sum_{\xi \in \multiplicativegroup{\left(\finiteFieldExtension{n} \otimes_\finiteField \finiteFieldExtension{m}\right)} } \alpha^{-1}\left( \etaleNorm_1^{n,m} \left(\xi\right) \right) \beta^{-1} \left( \etaleNorm_2^{n,m}\left( \xi \right) \right) \fieldCharacter \left( \trace \xi \right).
\end{equation}

If $\lambda = \left(n_1, \dots, n_r\right) \vdash n$, we denote $\finiteFieldExtension{\lambda} = \prod_{i=1}^r \finiteFieldExtension{n_i}$. In this case, $\multiplicativegroup{\finiteFieldExtension{\lambda}} = \prod_{i=1}^n \multiplicativegroup{\finiteFieldExtension{n_i}}$ and $\charactergroup{\lambda} = \prod_{i=1}^n \charactergroup{n_i}$. Then we have that $\finiteFieldExtension{\lambda} \otimes_{\finiteField} \finiteFieldExtension{m} \cong \prod_{i=1}^r \finiteFieldExtension{n_i} \otimes \finiteFieldExtension{m}$, and that for  $x = \left(x_1, \dots, x_r\right) \in \prod_{i=1}^r \finiteFieldExtension{n_i} \otimes_\finiteField \finiteFieldExtension{m}$, the norm maps are given by \begin{align*}
\etaleNorm_{1}^{\finiteFieldExtension{\lambda} \otimes_\finiteField \finiteFieldExtension{m}}\left(x\right) &= \left( \etaleNorm_{1}^{n_1, m}\left(x_1\right), \dots, \etaleNorm_{1}^{n_r, m}\left(x_r\right) \right),\\
\etaleNorm_{2}^{\finiteFieldExtension{\lambda} \otimes_\finiteField \finiteFieldExtension{m}}\left(x\right) &= \prod_{i=1}^r \etaleNorm_{2}^{n_i,m}\left(x\right),
\end{align*}
and the trace map (to $\finiteField$) is given by
$$ \trace x = \sum_{i=1}^r {\trace x_i}. $$

For $\alpha = \left(\alpha_1, \dots, \alpha_r\right) \in \charactergroup{\lambda} = \prod_{i = 1}^r{\charactergroup{n_i}}$ and $\beta \in \charactergroup{m}$, denote
$$\tau_{\lambda, m}\left(\alpha \times \beta, \fieldCharacter\right) = \left(-1\right)^{nm + n + rm} \sum_{\xi \in \multiplicativegroup{\left(\finiteFieldExtension{\lambda} \otimes_\finiteField \finiteFieldExtension{m}\right)}} \alpha^{-1}(\etaleNorm_1^{\finiteFieldExtension{\lambda} \otimes_\finiteField \finiteFieldExtension{m}} \left(\xi\right)) \beta^{-1}( \etaleNorm_{2}^{\finiteFieldExtension{\lambda} \otimes_\finiteField \finiteFieldExtension{m}}\left( \xi \right) ) \fieldCharacter\left(\trace \xi\right).$$
Then we have that
$$ \tau_{\lambda, m}\left(\alpha \times \beta, \fieldCharacter \right) = \prod_{j=1}^r \tau_{n_j, m}\left(\alpha_j \times \beta, \fieldCharacter\right).$$

\subsection{Katz's exotic Kloosterman sums and sheaves} \label{subsec:katz-exotic-kloosterman-sheaves}

Let $B$ be a finite étale $\finiteField$-algebra of rank $n$.
We consider the following diagram of schemes over $\finiteField$:

$$ \xymatrix{ & \restrictionOfScalars{B}{\finiteField}{\multiplcativeScheme} \ar[ld]_{ \mathrm{Norm} }  \ar[rd]^{ \mathrm{Trace} } \\
	\multiplcativeScheme & & \affineLine },$$
where $\mathrm{Norm}$ and $\mathrm{Trace}$ are the norm and trace maps, respectively. For a commutative $\finiteField$-algebra $R$, we have $\restrictionOfScalars{B}{\finiteField}{\multiplcativeScheme} \left(R\right) = \multiplicativegroup{\left(B \otimes_\finiteField R\right)}$, $\mathrm{Norm}\left(R\right) = \etaleNorm^{B \otimes_\finiteField R}_{2} \colon \multiplicativegroup{\left(B \otimes_\finiteField R\right)} \rightarrow \multiplicativegroup{R}$, and $\mathrm{Trace} \left(R\right) = \trace^{B \otimes_\finiteField R}_2 \colon B \otimes_\finiteField R \rightarrow R$.

Let $\ell$ be a prime different than the characteristic of $\finiteField$. Let $\fieldCharacter \colon \finiteField \rightarrow \multiplicativegroup{\ladicnumbers}$ be a non-trivial additive character, and let $\chi \colon \multiplicativegroup{B} \rightarrow \multiplicativegroup{\ladicnumbers}$ be a multiplicative character. The additive character $\fieldCharacter$ gives an Artin-Schrier local system $\artinScrier$ on $\affineLine$. $\chi$ gives a Kummer local system on $\restrictionOfScalars{B}{\finiteField}{\multiplcativeScheme}$: the Lang cover $\pi \colon \restrictionOfScalars{B}{\finiteField}{\multiplcativeScheme} \rightarrow \restrictionOfScalars{B}{\finiteField}{\multiplcativeScheme}$ is a $\restrictionOfScalars{B}{\finiteField}{\multiplcativeScheme}\left( \finiteField \right) = \multiplicativegroup{B}$-torsor, and $\mathcal{L}_{\chi}$ is the direct summand of $\pi_{\ast} \ladicnumbers$ on which $\multiplicativegroup{B}$ acts via $\chi$. Consider the following $\ell$-adic complex on $\multiplcativeScheme$:
$$ \operatorname{Kl}\left( B, \chi, \fieldCharacter \right) = \convolutionWithCompactSupport \mathrm{Norm}_{!} \left( \mathrm{Trace}^{\ast} \artinScrier \otimes \mathcal{L}_{\chi} \right) \left[n - 1\right].$$

This complex has been introduced and studied by Katz in \cite[Sections 8.8.4-8.8.7]{katz2016gauss}. See also \cite[Page 152]{katz1993estimates} and \cite[Appendix B]{nien2021converse}. It admits the following properties:

\begin{theorem}[{\cite[Theorem 8.8.5]{katz2016gauss}}]\label{thm:katz-etale-kloosterman-sheaf}
	Denote $\mathcal{K}= \operatorname{Kl}\left( B, \chi, \fieldCharacter \right)$. For $a \in \multiplicativegroup{\finiteField}$, let $\Frobenius_a \mid \mathcal{K}_a$ be the geometric Frobenius at $a$ acting on the stalk of $\mathcal{K}$. Then
	\begin{enumerate}
		\item \label{item:kloosterman-sheaf-is-of-rank-n} $\mathcal{K}$ is a local system of rank $n$. 
		\item \label{item:kloosterman-is-pure-of-weight-n-minus-1} $\mathcal{K}$ is pure of weight $n - 1$. That is, the eigenvalues of $\Frobenius_a \mid \mathcal{K}_a$ are algebraic integers, and given an embedding $\ladicnumbers \hookrightarrow \cComplex$, the eigenvalues of $\Frobenius_a \mid \mathcal{K}_a$ and all of their algebraic conjugates have absolute value $q^{\frac{n-1}{2}}$.
		\item \label{item:m-th-geometric-frobenius-trace-exotic-kloosterman-sum} The trace of the $m$-th power of the geometric Frobenius at $a$ acting on the stalk of $\mathcal{K}$ is given by
		$$ \trace\left( \Frobenius^m_a \mid \mathcal{K}_a \right) = \left(-1\right)^{n-1} \sum_{\substack{x \in B \otimes_\finiteField \finiteFieldExtension{m} \\
				\etaleNorm^{B \otimes_\finiteField \finiteFieldExtension{m}}_2(x) = a}} \chi ( \etaleNorm^{B\otimes_\finiteField \finiteFieldExtension{m}}_1 \left( x \right) ) \fieldCharacter\left( \trace x \right).$$
	\end{enumerate}
	
\end{theorem}

If $B = \finiteFieldExtension{\lambda}$, where $\lambda \vdash n$, and $\chi \colon \multiplicativegroup{\finiteFieldExtension{\lambda}} \rightarrow \multiplicativegroup{\ladicnumbers}$ is a multiplicative character (or $\chi \colon \multiplicativegroup{\finiteFieldExtension{\lambda}} \rightarrow \multiplicativegroup{\cComplex}$ and $\fieldCharacter \colon \finiteField \rightarrow \multiplicativegroup{\cComplex}$), and  $a \in \multiplicativegroup{\finiteField}$, we denote the sum appearing in \Cref{thm:katz-etale-kloosterman-sheaf} (\ref{item:m-th-geometric-frobenius-trace-exotic-kloosterman-sum}) by
$$ J_{m}\left(\chi,\fieldCharacter,a\right) = \sum_{\substack{x \in \finiteFieldExtension{\lambda} \otimes_\finiteField \finiteFieldExtension{m} \\
		\etaleNorm^{\finiteFieldExtension{\lambda} \otimes_\finiteField \finiteFieldExtension{m}}_2(x) = a}} \chi ( \etaleNorm^{\finiteFieldExtension{\lambda}\otimes_\finiteField \finiteFieldExtension{m}}_1 \left( x \right) ) \fieldCharacter\left( \trace x \right),$$
	and call $J_{m}\left(\fieldCharacter, \chi,a\right)$ an exotic Kloosterman sum. These sums generalize the unitary Kloosterman sums introduced in \cite{curtis1999unitary}.

\section{Formulas for special values of the Bessel function}\label{sec:formulas-for-special-values}

In this section, we develop a recursive expression for the Bessel function. We begin with defining a parametrization for irreducible generic representations of $\GL_{n}\left(\finiteField\right)$ using multiplicative characters compatible with a partition of $n$. This parametrization is defined even when the involved multiplicative characters are not regular. Then we rewrite the recursive expression of the Bessel function (\Cref{thm:recursive-bessel-using-gamma-factors}) using this parametrization. We relate special instances of this recursive expression to $L$-functions associated with exotic Kloosterman sums, and to the exotic Kloosterman sheaves constructed by Katz. As a corollary, we are able to show that certain polynomials with Bessel function values as their coefficients have roots lying on the unit circle. This gives us non-trivial bounds for special values of the Bessel function. Finally, we use Dickson polynomials in order to relate special values of the Bessel function of irreducible generic representations of $\GL_n\left(\finiteField\right)$ to special values of the Bessel function of their Shintani base change to $\GL_n\left(\finiteFieldExtension{k}\right)$.

\subsection{Generic representations of a given type}
Let $n \ge 1$, and let $\lambda = \left(n_1,\dots,n_r\right)$ be a partition of $n$.

Given multiplicative characters $\alpha_i \colon \multiplicativegroup{\finiteFieldExtension{n_i}} \rightarrow \multiplicativegroup{\cComplex}$ for every $1 \le i \le r$, we define an irreducible generic representation $\Pi_\lambda\left(\alpha_1,\dots,\alpha_r\right)$ as follows. Let $f_1, \dots, f_l$ be all the Frobenius orbits of $\alpha_1, \dots, \alpha_r$ without repetitions. We define $\phi = \phi_{\lambda}\left(\alpha_1, \dots, \alpha_r\right) \in P_n \left(\limitcharactergroup\right)$ by $\phi \left(f_i\right) = \left({s_i}\right)$, where $$s_i = \frac{1}{\frobeniusDegree \left(f_i\right)}\sum_{\substack{1 \le j \le r\\
		\alpha_j \in f_i}} n_j,$$ and $\left(\right)$ outside of $f_1, \dots, f_l$.
We denote $\Pi_\lambda \left(\alpha_1, \dots, \alpha_r \right) = \Pi_{\phi_\lambda \left(\alpha_1, \dots, \alpha_r \right)}$.

Any irreducible generic representation can be realized in the above form. We first express the tensor product $\epsilon_0$-factor of two irreducible generic representations given by this parametrization using Gauss sums.

\begin{proposition}\label{prop:epsilon-factor-expression}
	Let $\lambda = \left(n_1, \dots, n_r\right)$ and $\mu = \left(m_1, \dots, m_t\right)$ be partitions of $n$ and $m$, respectively. Then 
	\begin{equation}\label{eq:epsilon-factor-of-parameterized-generic-representation}
	\begin{split}
		&\epsilon_0\left(\Pi_{\lambda}\left(\alpha_1, \dots, \alpha_r\right) \times \Pi_{\mu}\left(\beta_1, \dots, \beta_t\right), \fieldCharacter \right) = \\
		&\left(-1\right)^{nm} q^{-\frac{nm}{2}} \prod_{i = 1}^r \prod_{j = 1}^t \prod_{k = 1}^{\gcd \left(n_i, m_j\right)}{\tau \left(\alpha_i \circ \FieldNorm{\lcm \left(n_i, m_j\right)}{n_i} \cdot \beta_j^{q^k} \circ \FieldNorm{{\lcm \left(n_i, m_j\right)}}{m_j}, \fieldCharacter_{\lcm \left(n_i, m_j\right)} \right)} 
	\end{split}
	\end{equation}
	for any choice of multiplicative characters $\alpha_i \colon \multiplicativegroup{\finiteFieldExtension{n_i}} \rightarrow \multiplicativegroup{\cComplex}$ and $\beta_j \colon \multiplicativegroup{\finiteFieldExtension{m_j}} \rightarrow \multiplicativegroup{\cComplex}$.
\end{proposition}
\begin{proof}
	Let $d_1 = \frobeniusDegree \left(\alpha_1\right), \dots, d_r = \frobeniusDegree \left(\alpha_r\right)$, and $u_1 = \frobeniusDegree \left(\beta_1\right), \dots, u_t = \frobeniusDegree \left(\beta_t\right)$ be the degrees of the Frobenius orbits generated by the multiplicative characters. Let $\alpha'_i \in \charactergroup{d_i}$ and $\beta'_j \in \charactergroup{u_j}$ be characters, such that $\alpha_i = \alpha'_i \circ \FieldNorm{n_i}{d_i}$ and $\beta_j = \beta'_j \circ \FieldNorm{m_j}{u_j}$. Since $$\alpha_i \circ \FieldNorm{\lcm \left(n_i, m_j\right)}{n_i} = \alpha'_i \circ \FieldNorm{\lcm \left(n_i, m_j\right)}{d_i} = \alpha'_i \circ \FieldNorm{\lcm\left(d_i, u_j\right)}{d_i} \circ \FieldNorm{\lcm \left(n_i, m_j\right)}{\lcm \left(d_i, u_j\right)},$$ and similarly, $$\beta_j^{q^k} \circ \FieldNorm{\lcm\left(n_i, m_j\right)}{m_j} = {\beta_j'}^{q^k} \circ \FieldNorm{\lcm\left(n_i, m_j\right)}{u_j} = {\beta_j'}^{q^k} \circ \FieldNorm{\lcm\left(d_i, u_j\right)}{u_j} \circ \FieldNorm{\lcm \left(n_i, m_j\right)}{\lcm \left(d_i, u_j\right)},$$ we have by the Hasse-Davenport relation that \begin{align*}
		 &\tau \left(\alpha_i \circ \FieldNorm{\lcm \left(n_i, m_j\right)}{n_i} \cdot \beta_j^{q^k} \circ \FieldNorm{{\lcm \left(n_i, m_j\right)}}{m_j}, \fieldCharacter_{\lcm \left(n_i, m_j\right)} \right) \\
		 =&\tau \left(\alpha'_i \circ \FieldNorm{\lcm \left(d_i, u_j\right)}{d_i} \cdot {\beta_j'}^{q^k} \circ \FieldNorm{{\lcm \left(d_i, u_j\right)}}{u_j}, \fieldCharacter_{\lcm (d_i, u_j)} \right)^{\frac{\lcm (n_i, m_j)}{\lcm (d_i, u_j)}}.
	\end{align*}
	We notice that the product of the Gauss sums satisfies \begin{equation*}
		\begin{split}
		&\prod_{k = 1}^{\gcd \left(n_i, m_j\right)}{\tau \left(\alpha_i \circ \FieldNorm{\lcm \left(n_i, m_j\right)}{n_i} \cdot {\beta}_j^{q^k} \circ \FieldNorm{{\lcm \left(n_i, m_j\right)}}{m_j}, \fieldCharacter_{\lcm \left(n_i, m_j\right)} \right)}  \\=
		&\left(\prod_{k = 1}^{\gcd \left(d_i, u_j\right)}{\tau \left(\alpha'_i \circ \FieldNorm{\lcm \left(d_i, u_j\right)}{d_i} \cdot {\beta_j'}^{q^k} \circ \FieldNorm{{\lcm \left(d_i, u_j\right)}}{u_j}, \fieldCharacter_{\lcm \left(d_i, u_j\right)} \right)^{\frac{\lcm (n_i, m_j)}{\lcm (d_i, u_j)}}}\right)^{\frac{\gcd(n_i, m_j)}{\gcd(d_i, u_j)}},
		\end{split}
	\end{equation*}
	as the multiplicand repeats itself every $\gcd\left(d_i, u_j\right)$ terms (see \Cref{prod:tensor-product-gauss-sum-invariant-under-gcd}). Hence we have that 
	\begin{equation}\label{eq:gauss-sum-product-factorization-through-gcd}
	\begin{split}
	&\prod_{k = 1}^{\gcd \left(n_i, m_j\right)}{\tau \left(\alpha_i \circ \FieldNorm{\lcm \left(n_i, m_j\right)}{n_i} \cdot {\beta}_j^{q^k} \circ \FieldNorm{{\lcm \left(n_i, m_j\right)}}{m_j}, \fieldCharacter_{\lcm \left(n_i, m_j\right)} \right)} \\=
	&	\left(\prod_{k=1}^{\gcd\left(d_i, u_j\right)} \tau \left(\alpha'_i \circ \FieldNorm{\lcm \left(d_i, u_j\right)}{d_i} \cdot {\beta_j'}^{q^k} \circ \FieldNorm{{\lcm \left(d_i, u_j\right)}}{u_j}, \fieldCharacter_{\lcm \left(d_i, u_j\right)} \right)\right)^{\frac{n_i m_j}{d_i u_j}}.
	\end{split}
	\end{equation}
	Denote by $f_1, \dots, f_l$ the different Frobenius orbits generated by $\alpha'_1, \dots, \alpha'_r$, and by $g_1, \dots, g_{s}$ the different Frobenius orbits generated by $\beta'_1, \dots, \beta'_{t}$. By \eqref{eq:gauss-sum-product-factorization-through-gcd}, \Cref{thm:epsilon-factor-of-tensor-product-as-product-of-gauss-sums}, and using the fact that $nm = \sum_{i'=1}^r \sum_{j'=1}^t n_{i'} m_{j'}$, we can rewrite the right hand side of \eqref{eq:epsilon-factor-of-parameterized-generic-representation} as \begin{align*}
		&\prod_{i = 1}^{l}\prod_{j = 1}^{s}\prod_{\substack{1 \le i' \le r\\
		\alpha_{i'} \in f_i}}\prod_{\substack{1 \le j' \le t\\
	\beta_{j'} \in g_j
}} {\epsilon_0 \left( \Pi_{f_i} \times \Pi_{g_j} , \fieldCharacter \right)^{\frac{n_{i'} m_{j'}}{\frobeniusDegree (f_i)  \frobeniusDegree(g_j)}}} \\
	=& \prod_{i = 1}^{l}\prod_{j = 1}^{s} {\epsilon_0 \left( \Pi_{f_i} \times \Pi_{g_j} , \fieldCharacter \right)^{\left(\sum_{\substack{1 \le i' \le r\\
						\alpha_{i'} \in f_i}} \frac{n_{i'}}{\frobeniusDegree (f_i)}\right) \left(\sum_{\substack{1 \le j' \le t\\
						\beta_{j'} \in g_j}} \frac{m_{j'}}{\frobeniusDegree(g_j)}\right) }}.
	\end{align*}
	By multiplicativity of the tensor product $\epsilon_0$-factors 
	(\Cref{thm:epsilon-factors-are-multiplicative}), and by the definitions of $\Pi_{\lambda} \left( \alpha_1, \dots, \alpha_r \right)$ and $\Pi_{\mu} \left( \beta_1, \dots, \beta_t \right)$, this product equals $$\epsilon_0\left(\Pi_{\lambda}\left(\alpha_1, \dots, \alpha_r\right) \times \Pi_{\mu}\left(\beta_1, \dots, \beta_t\right), \fieldCharacter \right),$$ as required.
\end{proof}

Using the formalism of étale algebras discussed in \Cref{sec:tensor-product-of-finite-fields}, we can restate \Cref{prop:epsilon-factor-expression} in the following way.

\begin{proposition} \label{prop:epsilon-factors-etale-product}
	Let $\lambda = \left(n_1, \dots, n_r\right)$ and $\mu = \left(m_1, \dots, m_t\right)$ be partitions of $n$ and $m$, respectively. Then 
	\begin{equation*}
	\begin{split}
	\epsilon_0\left(\Pi_{\lambda}\left(\alpha_1, \dots, \alpha_r\right) \times \Pi_{\mu}\left(\beta_1, \dots, \beta_t\right), \fieldCharacter \right) &= \left(-1\right)^{nm}q^{-\frac{nm}{2}}\prod_{j=1}^t \tau_{\lambda, m_j}\left(\alpha \times \beta_j, \fieldCharacter \right),
	\end{split}
	\end{equation*}
	for any choice of multiplicative characters $\alpha = \left(\alpha_1, \dots, \alpha_r \right) \in \charactergroup{\lambda} = \prod_{i=1}^r\charactergroup{n_i}$ and $\beta = \left(\beta_1, \dots , \beta_t\right) \in \charactergroup{\mu} = \prod_{j=1}^t \charactergroup{m_j}$.
\end{proposition}

Denote for a partition $\mu = \left(m_1, \dots, m_t\right) \vdash m$ the polynomial $\varphi_{\mu}\left(t\right) = {\prod_{j=1}^t\left(t^{m_j} - 1\right)}$ and $Z_\mu = \prod_{k = 1}^{\infty}{k^{\mu\left(k\right)} \cdot \mu\left(k\right)!}$, where $\mu\left(k\right)$ is the number of times $k$ appears in the partition $\mu$, i.e., $\mu \left(k\right) = \#\left\{ 1 \le j \le t \mid m_j = k \right\}$.

The following lemma will be used in order to reduce the summation over irreducible generic representations in \Cref{thm:recursive-bessel-using-gamma-factors} to a summation over multiplicative characters.
\begin{lemma}\label{lem:number-of-preimages-of-partition-sequence}
	Let $\sigma$ be an irreducible generic representation of $\GL_m\left(\finiteField\right)$. Suppose that $\sigma$ is parameterized by $\phi \in P_m\left(\limitcharactergroup\right)$, which is supported on the pairwise distinct Frobenius orbits $f_1, \dots, f_r$ of degrees $d_1 = \frobeniusDegree \left(f_1\right), \dots, d_r = \frobeniusDegree \left(f_r\right)$, respectively, and suppose that $\phi\left(f_i\right) = \left({s_i}\right)$. 
	Let $$\Omega_\sigma = \left\{ \left(\left(\nu_1, \dots, \nu_r\right), \left(\gamma_1, \dots, \gamma_r\right) \right) \mid \nu_i \vdash s_i, \gamma_i \in f_i^{\lengthof\left( \nu_i \right)} \right\},$$
	where $\lengthof\left(\nu_i\right)$ is the length of the partition $\nu_i$, and $f_i^{\lengthof\left(\nu_i\right)} = f_i \times \dots \times f_i$ is the Cartesian product set, where $f_i$ appears in the product $\lengthof\left(\nu_i\right)$ times. Then
	\begin{enumerate}
		\item \label{item:surjective-parameterization} There exists a surjective map $$\eta_\sigma \colon \left\{ \left(\mu, \beta \right) \mid \mu \vdash m, \beta \in \charactergroup{\mu} \mid \Pi_{\mu}\left(\beta\right) \isomorphic \sigma \right\} \rightarrow \Omega_\sigma.$$
		\item Let \label{item:partition-of-inverse-image-of-parameterization} $\left( \left(\nu_1,\dots,\nu_r\right), \left(\gamma_1,\dots,\gamma_r\right) \right) \in \Omega_\sigma$. Let $\mu$ be the concatenation of the partitions $d_1 \nu_1,\dots, d_r \mu_r$. Then for any $\mu' \vdash m$ and $\beta' \in \charactergroup{\mu'}$ such that $\eta_\sigma \left(\mu', \beta'\right) = \left( \left(\nu_1,\dots,\nu_r\right), \left(\gamma_1,\dots,\gamma_r\right) \right)$, we have $\mu' = \mu$ and $\varphi_{\mu'}\left(q\right) = \prod_{i=1}^r{\varphi_{\nu_i}\left(q^{d_i}\right)}$.
		\item \label{item:size-of-inverse-image-of-parameterization} The inverse image of $\left( \left(\nu_1,\dots,\nu_r\right), \left(\gamma_1,\dots,\gamma_r\right) \right)$ under $\eta_\sigma$ is of cardinality $$\frac{Z_\mu}{ d_1^{\lengthof\left( \nu_1 \right)} Z_{\nu_1} \dots d_r^{\lengthof\left( \nu_r \right)} Z_{\nu_r}},$$
		where $\mu$ is defined in part (\ref{item:partition-of-inverse-image-of-parameterization}).

	\end{enumerate}
	\begin{proof}
		Let $\mu = \left(m_1,\dots,m_t\right) \vdash m$ and $\beta = \left(\beta_1, \dots, \beta_t\right) \in \charactergroup{\mu}$ with $\Pi_{\mu}\left(\beta\right) \isomorphic \sigma$. For every $i$, consider the index set $J_i = \left\{1 \le j \le t \mid \beta_j \in f_i \right\}$. Then by the definition of $\Pi_{\mu}\left(\beta\right)$ we must have $$\sum_{j \in J_i} \frac{m_{j}}{d_i} = s_i,$$ and therefore the tuple $$\nu_i = \left(\frac{m_{j}}{d_i} \mid j \in J_i \right)$$ is a partition of $s_i$, where $\nu_i$ is ordered using the ordering of $J_i$. Denote $$\gamma_i = \left(\beta_{j} \mid j \in J_i \right) \in f_i^{\lengthof\left(\nu_i\right)},$$
		where again $\gamma_i$ is ordered using the ordering of $J$.
		We define $$\eta_\sigma\left(\mu, \beta \right) = \left(\left(\nu_1,\dots,\nu_r\right),\left(\gamma_1,\dots,\gamma_r\right)\right) \in \Omega_\sigma.$$ Therefore, we have constructed a map as in the theorem.
		
		We move to show that $\eta_\sigma$ is surjective. Given $\nu_1 \vdash s_1, \dots, \nu_r \vdash s_r$ and $\gamma_i \in f_i^{\lengthof\left(\nu_i\right)}$, write $\nu_i = \left( s_{i,1}, \dots, s_{i,l_i} \right)$ and $\gamma_i = \left(\gamma_{i,1}, \dots, \gamma_{i,l_i} \right)$.
		Denote $\mu_i = d_i \nu_i = \left( d_i s_{i,1}, \dots, d_i s_{i,l_i} \right)$, and denote by $\mu$ the concatenation of $\mu_1, \dots, \mu_r$. Let $\beta_{i,j} \in \charactergroup{d_i}$ be a multiplicative character representing $\gamma_{i,j}$, and let $$\beta_i = \left( \beta_{i,1} \circ \FieldNorm{d_i s_{i,1}}{d_i}, \dots, \beta_{i,l_i} \circ \FieldNorm{d_i s_{i,l_i}}{d_i} \right).$$
		Let $\beta$ be the concatenation of $\beta_1, \dots, \beta_r$, where we perform the concatenation as we did for $\mu_1,\dots,\mu_r$, so that the indices of both concatenations match. Then $\Pi_{\mu}\left(\beta\right) \isomorphic \sigma$ by definition and  $$\eta_\sigma\left(\mu, \beta\right) = \left(\left(\nu_1,\dots,\nu_r\right),\left(\gamma_1,\dots,\gamma_r\right)\right).$$ This proves part (\ref{item:surjective-parameterization}).
		
		If $\left(\mu', \beta'\right)$ is an inverse image of $\left(\left(\nu_1,\dots,\nu_r\right), \left(\gamma_1,\dots,\gamma_r\right)\right)$, then by our construction we must have that $\mu' = \mu$, where $\mu$ is constructed as above, that is, $\mu$ is the concatenation of $d_1 \nu_1, \dots, d_r \nu_r$. 		The identity $\varphi_{\mu}\left(q\right) = \prod_{i=1}^r\varphi_{\nu_i}\left(q^{d_i}\right)$ is immediate from the definitions of $\mu$ and $\varphi_\mu \left(t\right)$. This proves part (\ref{item:partition-of-inverse-image-of-parameterization}).
		
		We move to count the number of options for the sequence $\beta' = (\beta_1', \dots, \beta_{\lengthof\left(\mu\right)}')$. The sequence $\beta'$, up to ordering, has to be the same as the $\beta$ we constructed above. We notice that for every $i$, once we know the indices of all parts of $\mu$ which are occupied by elements lying in the orbit $f_i$, then the value of $\beta'$ at these indices is determined by $\gamma_i$. Therefore, the only choice for our ordering is how to label the parts of $\mu$ with $f_1$, $\dots$, $f_r$, such that the parts labeled $f_i$ form the partition $d_i \nu_i$, for every $i$. In order to compute this, we first count number of ways to label the parts of size $k$ in $\mu$ by $f_1$, $\dots$, $f_r$, so that for every $i$, the number of parts of size $k$ in $\mu$ labeled $f_i$ equals the number of the parts of size $k$ in $d_i \nu_i$. The number of way doing so is
		$$ \binom{\mu\left(k\right)}{\left(d_1 \nu_1\right)\left(k\right), \dots, \left(d_r \nu_r\right)\left(k\right)},$$
		where for every $1 \le i \le r$, $\left(d_i \nu_i\right)\left(k\right)$ is the number of times $k$ appears in the partition $d_i \nu_i$.
		
		Therefore, we have $$\prod_{k=1}^{\infty}{ \binom{\mu\left(k\right)}{\left(d_1 \nu_1\right)\left(k\right), \dots, \left(d_r \nu_r\right)\left(k\right)} } = \prod_{k=1}^{\infty} \frac{\mu\left(k\right)!}{\left(\left(d_1 \nu_1\right)\left(k\right)\right)! \dots  \left(\left(d_1 \nu_r\right)\left(k\right)\right)!} = \prod_{k=1}^{\infty} \frac{\mu\left(k\right)!}{\nu_1\left(k\right)! \dots  \nu_r\left(k\right)!}$$ options for labeling $\mu$ with $f_1, \dots, f_r$, such that for every $i$, the parts labeled $f_i$ form the partition $d_i \nu_i$. This number equals $$\prod_{k=1}^{\infty} \frac{\left(d_1 k\right)^{\nu_1\left(k\right)} \dots \left(d_r k\right)^{\nu_r\left(k\right)} \mu\left(k\right)!}{d_1^{\nu_1\left(k\right)} \dots d_r^{\nu_r\left(k\right)} k^{\nu_1\left(k\right)} \dots k^{\nu_r\left(k\right)} \nu_1\left(k\right)! \dots  \nu_r\left(k\right)!},$$ which equals $$\frac{Z_\mu}{d_1^{\lengthof\left(\nu_1\right)} \dots d_r^{\lengthof\left(\nu_r\right)} Z_{\nu_1} \dots Z_{\nu_r}},$$ since $\sum_{k=1}^{\infty}\nu_i\left(k\right) = \lengthof\left(\nu_i\right)$ and since by definition $$\prod_{k=1}^{\infty} k^{\mu\left(k\right)} = \prod_{k=1}^{\infty} \left(d_1 k\right)^{\nu_1\left(k\right)} \dots \left(d_r k\right)^{\nu_r\left(k\right)}.$$
		This proves part (\ref{item:size-of-inverse-image-of-parameterization}).
	\end{proof}
\end{lemma}

We will now use this lemma in order to relate the dimension of an irreducible generic representation $\sigma$ of $\GL_m\left(\finiteField\right)$ to the number of $\mu \vdash m$ and $\beta \in \charactergroup{\mu}$ satisfying $\Pi_\mu\left(\beta\right) \cong \sigma$.

\begin{theorem}
	Let $\representationDeclaration{\sigma}$ be an irreducible generic representation of $\GL_{m}\left(\finiteField\right)$. Then 
	$$ \frac{1}{\grpIndex{\GL_m\left(\finiteField\right)}{\UnipotentSubgroup_m}} {{\dim \sigma}} = \sum_{\mu \vdash m} \frac{1}{Z_\mu} \frac{1}{\varphi_\mu\left(q\right)} \cdot \#\left\{\beta \in \charactergroup{\mu} \mid \Pi_{\mu}\left(\beta\right) \cong \sigma \right\}. $$
\end{theorem}
\begin{proof}
	Suppose that $\sigma$ is parameterized by $\phi \in P_m\left(\limitcharactergroup\right)$, which is supported on the pairwise distinct Frobenius orbits $f_1, \dots, f_r$ of degrees $d_1 = \frobeniusDegree \left(f_1\right), \dots, d_r = \frobeniusDegree \left(f_r\right)$, respectively, and suppose that $\phi\left(f_i\right) = \left({s_i}\right)$. In this case we have by \eqref{eq:dimension-of-generic-representation}, $$\dim \sigma = \grpIndex{\GL_{m}\left(\finiteField\right)}{\UnipotentSubgroup_m} \prod_{i=1}^{r}\frac{q^{\frac{d_i s_i \left(s_i - 1\right)}{2}}}{\prod_{j=1}^{s_i}\left(q^{d_i j} - 1\right)}.$$
	By \Cref{lem:number-of-preimages-of-partition-sequence}, we have \begin{align*}
	\sum_{\mu \vdash m} \sum_{\substack{\beta \in \charactergroup{\mu}\\
			\Pi_{\mu}\left(\beta\right) \isomorphic \sigma}} \frac{1}{Z_\mu} \frac{1}{\varphi_\mu\left(q\right)}
	= \sum_{\nu_1 \vdash s_1, \dots, \nu_r \vdash s_r}  \sum_{\gamma_1 \in f_1^{\lengthof\left(\nu_1\right)}, \dots, \gamma_r \in f_r^{\lengthof\left(\nu_r\right)}}\frac{1}{Z_{\mu}} \frac{1}{\varphi_{\mu}\left(q\right)} \frac{Z_\mu}{d_1^{\lengthof\left(\nu_1\right)} \dots d_r^{\lengthof\left(\nu_r\right)} Z_{\nu_1} \dots Z_{\nu_r}},
	\end{align*}where on the right hand side $\mu = \mu\left(\nu_1,\dots,\nu_r\right)$ is defined in \Cref{lem:number-of-preimages-of-partition-sequence} and satisfies $\varphi_{\mu}\left(q\right) = \prod_{i=1}^r{ \varphi_{\nu_i}\left(q^{d_i}\right) }$. Since $\#f_i^{\lengthof\left(\nu_i\right)} = d_i^{\lengthof\left(\nu_i\right)}$, we need to show that $$\prod_{i=1}^{r}\frac{q^{\frac{d_i s_i \left(s_i - 1\right)}{2}}}{\prod_{j=1}^{s_i}\left(q^{d_i j} - 1\right)} = \sum_{\nu_1 \vdash s_1}\dots\sum_{\nu_r \vdash s_r} \frac{1}{\prod_{i=1}^r{Z_{\nu_i} \varphi_{\nu_i}\left(q^{d_i}\right)}}.$$
	
	By Equation (2.14') of \cite[Page 25]{macdonald1998symmetric}, applied to Example 5 on \cite[Page 27]{macdonald1998symmetric}, we have the following identity of rational functions:
	$$\frac{q^{\frac{s \left(s-1\right)}{2}}}{\prod_{k=1}^s\left(q^k - 1\right)} = \sum_{\lambda \vdash s} \frac{1}{Z_\lambda\varphi_{\lambda}\left(q\right)},$$
	which implies our identity, by substituting $q \mapsto q^{d_i}$ and $s \mapsto s_i$ for every $i$, and multiplying these equalities.
\end{proof}

As a corollary, we obtain the following result, which gives an alternative formula for the expression in \Cref{thm:recursive-bessel-using-gamma-factors}.

\begin{theorem}\label{prop:bessel-function-recursive-using-characters}Let $\pi$ be an irreducible generic representation of $\GL_n\left(\finiteField\right)$. Then
	$$\mathcal{F}_{\pi, m, \fieldCharacter}\left(g\right) = \sum_{\mu \vdash m} \frac{1}{Z_\mu} \frac{1}{\varphi_\mu \left(q\right)} \sum_{\beta \in \charactergroup{\mu}} {\gamma\left(\pi \times \Contragradient{\Pi_{\mu}\left(\beta\right)}, \fieldCharacter\right)}  \cdot \besselFunction_{\Pi_{\mu}\left(\beta\right), \fieldCharacter}\left(g\right).$$
	In particular, we have the following special cases:
	\begin{enumerate}
		\item If $m < n$, then \begin{align*}
		\besselFunction_{\pi, \fieldCharacter}\begin{pmatrix}
		& \IdentityMatrix{n - m}\\	
		g &
		\end{pmatrix} =& \sum_{\mu \vdash m} \frac{1}{Z_\mu} \frac{1}{\varphi_\mu \left(q\right)} \sum_{\beta \in \charactergroup{\mu}} \gamma\left(\pi \times \Contragradient{\Pi_{\mu}\left(\beta\right)},\fieldCharacter\right) \cdot \besselFunction_{\Pi_{\mu}\left(\beta\right), \fieldCharacter}\left(g\right).
		\end{align*}
		\item If $m = n$, then
		\begin{align*}
		\besselFunction_{\pi, \fieldCharacter}\left(g\right)\fieldCharacter \begin{pmatrix}
		I_n & g^{-1}\\
		& I_n
		\end{pmatrix} =& \sum_{\mu \vdash m} \frac{1}{Z_\mu} \frac{1}{\varphi_\mu \left(q\right)} \sum_{\beta \in \charactergroup{\mu}} \gamma\left(\pi \times \Contragradient{\Pi_{\mu}\left(\beta\right)},\fieldCharacter\right) \cdot \besselFunction_{\Pi_{\mu}\left(\beta\right), \fieldCharacter}\left(g\right).
		\end{align*}
		\item If $m > n$ and $c \in \multiplicativegroup{\finiteField}$, then
		\begin{align*}
		\sum_{\mu \vdash m} \frac{1}{Z_\mu} \frac{1}{\varphi_\mu \left(q\right)} \sum_{\beta \in \charactergroup{\mu}} \gamma\left(\pi \times \Contragradient{\Pi_{\mu}\left(\beta\right)},\fieldCharacter\right) \cdot \besselFunction_{\Pi_{\mu}\left(\beta\right), \fieldCharacter}\left(c \IdentityMatrix{m} \right) &= 0.
\end{align*}
		
	\end{enumerate}	
	
\end{theorem}

Combining this with \Cref{thm:relation-between-epsilon-factors-and-gamma-factors}, and using the fact that for $\beta \in \charactergroup{\mu}$ and $c \in \multiplicativegroup{\finiteField}$ we have $\besselFunction_{\Pi_{\mu}\left(\beta\right), \fieldCharacter}\left(c \IdentityMatrix{m}\right) = \beta\left(c\right)$, we get the following formula.

\begin{corollary} \label{cor:recursive-formula-for-bessel-with-epsilon}
	Let $\representationDeclaration{\pi}$ be an irreducible generic representation of $\GL_{n}\left(\finiteField\right)$, where $n > 1$. Then for any $m \ge 1$ and $c \in \multiplicativegroup{\finiteField}$, we have
	\begin{align*}
	& q^{ -\frac{m \left(n - m - 1\right)}{2} }  \sum_{\mu \vdash m} \frac{1}{Z_\mu} \frac{1}{\varphi_\mu \left(q\right)} \sum_{\beta \in \charactergroup{\mu}} \beta \left(\left(-1\right)^{n-1} c \right)\epsilon_0 \left(\pi \times \Contragradient{\Pi_{\mu}\left(\beta\right)},\fieldCharacter\right) \\
	=	& \begin{dcases}
	 \besselFunction_{\pi, \fieldCharacter}\begin{pmatrix}
	& \IdentityMatrix{n - m}\\	
	c I_m &
	\end{pmatrix}  & m \le n,\\
	0 & m > n.
	\end{dcases}
	\end{align*}
\end{corollary}

\subsection{Relation to exotic Kloosterman sums and sheaves}

We now relate the result of \Cref{prop:bessel-function-recursive-using-characters} to Katz's exotic Kloosterman sums and sheaves.

Let $n > 1$. Suppose that $\pi = \Pi_{\lambda}\left( \alpha \right)$, where $\lambda = \left(n_1,\dots,n_r\right) \vdash n$ and $\alpha = \left(\alpha_1,\dots,\alpha_r\right) \in \charactergroup{\lambda} = \prod_{i=1}^r{\charactergroup{n_i}}$.

The $L$-function associated with the family of exotic Kloosterman sums $\left(J_{m}\left( \alpha, \fieldCharacter, a \right)\right)_{m=1}^\infty$, introduced in \Cref{subsec:katz-exotic-kloosterman-sheaves}, is defined as the following formal power series in the variable $T$.
$$ L\left(T, J\left(\alpha, \fieldCharacter, a\right)\right) = \exp\left(\sum_{m = 1}^{\infty} J_{m}\left(\chi, \fieldCharacter, a\right) \frac{T^m}{m} \right).$$
See also the introduction of \cite{fu2005functions}. We denote the normalized version $$L^{\ast}\left(T, J\left(\alpha, \fieldCharacter, a\right)\right) = L\left( \left(-1\right)^r q^{-\frac{\left(n-1\right)}{2}} T, J\left(\alpha, \fieldCharacter, a\right)\right)^{\left(-1\right)^n}.$$

Our first result expresses this $L$-function as a polynomial with special values of the Bessel function as its coefficients.

\begin{theorem}\label{thm:L-function-of-exotic-kloosterman-sum}
	For any $c \in \multiplicativegroup{\finiteField}$, we have $$L\left(T, J\left(\alpha^{-1}, \fieldCharacter, \left(-1\right)^{n-1} c^{-1}\right) \right)^{\left(-1\right)^n} = \sum_{m = 0}^{n}  \left(\left(-1\right)^{r} q^{\frac{2n-m-1}{2}}\right)^m \besselFunction_{\pi, \fieldCharacter}\begin{pmatrix}
	& \IdentityMatrix{n-m}\\
	c \IdentityMatrix{m}
	\end{pmatrix}  T^m,$$
	or, written in normalized form,
$$L^{\ast}\left(T, J\left(\alpha^{-1}, \fieldCharacter, \left(-1\right)^{n-1} c^{-1}\right) \right) = \sum_{m = 0}^{n}  q^{\frac{m\left(n-m\right)}{2}} \besselFunction_{\pi, \fieldCharacter}\begin{pmatrix}
& \IdentityMatrix{n-m}\\
c \IdentityMatrix{m}
\end{pmatrix}  T^m.$$	
\end{theorem}

\begin{proof}
	Using the Taylor series of the exponential function, we have that \begin{align*}
		&L\left(T, J_\lambda\left(\alpha^{-1}, \fieldCharacter, \left(-1\right)^{n-1} c^{-1} \right)\right)^{\left(-1\right)^n} \\
		=& \sum_{m = 0}^{\infty} \sum_{\mu = \left(m_1, \dots, m_t\right) \vdash m} \frac{1}{Z_\mu} \left(\prod_{j=1}^t \left(-1\right)^{n} J_{m_j}\left(\alpha^{-1}, \fieldCharacter, \left(-1\right)^{n-1} c^{-1} \right)\right) T^m.
	\end{align*}

	By \Cref{cor:recursive-formula-for-bessel-with-epsilon}, it suffices to show that for every $m$,	
	\begin{align*}
		& \left(-1\right)^{rm} q^{ \frac{nm}{2} }  \sum_{\mu \vdash m} \frac{1}{Z_\mu} \frac{1}{\varphi_\mu \left(q\right)} \sum_{\beta \in \charactergroup{\mu}} \beta \left(\left(-1\right)^{n-1} c \right)\epsilon_0 \left(\pi \times \Contragradient{\Pi_{\mu}\left(\beta\right)},\fieldCharacter\right) \\
		&=
		\sum_{\mu = \left(m_1, \dots, m_t\right) \vdash m} \frac{1}{Z_\mu}
		\left(-1\right)^{nt} \left(\prod_{j=1}^t  J_{m_j}\left(\alpha^{-1}, \fieldCharacter, \left(-1\right)^{n-1} c^{-1} \right)\right).
	\end{align*}
	
	Using \Cref{prop:epsilon-factors-etale-product}, we have that for $\mu = \left(m_1,\dots,m_t\right) \vdash m$, \begin{align*}
	&\sum_{\beta \in \charactergroup{\mu}} \epsilon_0 \left(\Pi_{\lambda}\left(\alpha\right) \times \Contragradient{\Pi_{\mu}\left(\beta\right)},\fieldCharacter\right) \beta \left(\left(-1\right)^{n-1} c\right) \\
	&= \left(-1\right)^{nm} q^{-\frac{nm}{2}} \sum_{\beta_1 \in \charactergroup{m_1}} \dots \sum_{\beta_t \in \charactergroup{m_t}} \prod_{j=1}^t \beta_j \left(\left(-1\right)^{n-1} c\right) \tau_{\lambda, m_j} \left(\alpha \times \beta_j^{-1}, \fieldCharacter \right).
	\end{align*}
	Using the fact that the sum of a non-trivial character over a group is zero, and the fact that $\varphi_\mu\left(q\right) = \sizeof{\charactergroup{\mu}}  = \prod_{j=1}^t \sizeof{\charactergroup{m_j}}$, we get that
	\begin{align*}
	&\frac{1}{\varphi_\mu \left(q\right)} \sum_{\beta \in \charactergroup{\mu}} \epsilon_0 \left(\Pi_{\lambda}\left(\alpha\right) \times \Contragradient{\Pi_{\mu}\left(\beta\right)},\fieldCharacter\right) \beta \left(\left(-1\right)^{n-1} c\right) \\
	& = \left(-1\right)^{nt + mr} q^{-\frac{nm}{2}} \prod_{j=1}^t J_{m_j} \left( \alpha^{-1}, \fieldCharacter, \left(-1\right)^{n-1}c^{-1} \right),
	\end{align*}
	as required.
\end{proof}

As a result, we get the following corollary, which can be seen as a generalization of the results of \cite{curtis1999unitary}.
\begin{corollary} \label{cor:hasse-davenport-for-exotic-kloosterman-sums}
	Consider the polynomial $$ C\left(X\right) = \sum_{m = 0}^n  \left(-1\right)^{rm} q^{\frac{m \left(2n - m - 1\right)}{2}} \besselFunction_{\pi, \fieldCharacter}\begin{pmatrix}
	& \IdentityMatrix{n-m}\\	
	c\IdentityMatrix{m} &
	\end{pmatrix} X^m.$$
	Write $$C\left(X\right) = \prod_{i=1}^n {\left(1 - \omega_i X\right)},$$
	where $\omega_1, \dots, \omega_n \in \cComplex$.  Then for every $m$,
	$$ \omega_1^m + \dots + \omega_n^m = \left(-1\right)^{n-1} J_{m}\left(\alpha^{-1}, \fieldCharacter, \left(-1\right)^{n-1} c^{-1} \right).$$
\end{corollary}
\begin{proof}
	Use \Cref{thm:L-function-of-exotic-kloosterman-sum} to write $C\left(X\right) = L\left(X, J\left(\alpha^{-1}, \fieldCharacter, \left(-1\right)c^{-1}\right)\right)^{\left(-1\right)^n}$. Taking logarithmic derivatives of both sides, we have
	$$\sum_{m = 1}^{\infty} \left(-1\right)^n J_m\left( \alpha^{-1}, \fieldCharacter, \left(-1\right)^{n-1} c^{-1} \right) X^{m-1} = -\sum_{i = 1}^n \frac{\omega_i}{1 - \omega_i X}.$$
	The result now follows by expanding the right hand side to a sum of geometric series and comparing coefficients.
\end{proof}

We also obtain a functional equation for the $L$-function immediately.
\begin{corollary}
	We have \begin{align*}
		& L\left(T ,J\left(\alpha^{-1}, \fieldCharacter, \left(-1\right)^{n-1} c^{-1}\right) \right)^{\left(-1\right)^n} \\
		&= \left(-1\right)^{nr} \alpha\left(c\right)  q^{\frac{n \left(n-1\right)}{2}} T^n L \left(  q^{- \left(n-1\right)} T^{-1}, J\left(\alpha, \fieldCharacter^{-1}, \left(-1\right)^{n-1} c^{-1}\right) \right)^{\left(-1\right)^n},
	\end{align*}
	or, written in normalized form,
\begin{equation*}
L^{\ast}\left(T ,J\left(\alpha^{-1}, \fieldCharacter, \left(-1\right)^{n-1} c^{-1}\right) \right) = \alpha\left(c\right) T^n L^{\ast}\left( T^{-1}, J\left(\alpha, \fieldCharacter^{-1}, \left(-1\right)^{n-1} c^{-1}\right) \right).
\end{equation*}		
\end{corollary}
\begin{proof}
	This follows from \Cref{thm:L-function-of-exotic-kloosterman-sum} combined with \Cref{prop:complex-conjugate-of-bessel-function}.
\end{proof}

Our next result relates Katz's exotic Kloosterman sheaves to special values of the Bessel function. Fix a prime $\ell$ different from the characteristic of $\finiteField$, and fix an embedding $\ladicnumbers \hookrightarrow \cComplex$. Using this embedding, the characters $\fieldCharacter$ and $\alpha$ give rise to characters $\fieldCharacter \colon \finiteField \rightarrow \multiplicativegroup{\ladicnumbers}$ and $\alpha \colon \multiplicativegroup{\finiteFieldExtension{\lambda}} \rightarrow \multiplicativegroup{\ladicnumbers}$.

\begin{theorem}\label{thm:bessel-function-as-exterior-power}
	Let $\mathcal{K} = \operatorname{Kl}\left(\finiteFieldExtension{\lambda}, \alpha^{-1}, \fieldCharacter\right)$. Then for any $c \in \multiplicativegroup{\finiteField}$ and any $0 \le m \le n$,
	\begin{equation}\label{eq:bessel-function-as-exterior-power}
		\left(\left(-1\right)^{r-1} q^{ -\frac{\left(n-1\right)}{2} }\right)^m \trace  \left( \Frobenius_{ \left(-1\right)^{n-1} c^{-1} } \mid \wedge^m \mathcal{K}_{ \left(-1\right)^{n-1} c^{-1} }  \right) = q^{\frac{m \left(n-m\right)}{2}} \besselFunction_{\pi, \fieldCharacter}\begin{pmatrix}
		& \IdentityMatrix{n - m}\\	
		c \IdentityMatrix{m} &
		\end{pmatrix},
	\end{equation}
	where $\Frobenius_{ \left(-1\right)^{n-1} c^{-1} } \mid \wedge^m \mathcal{K}_{ \left(-1\right)^{n-1} c^{-1} }$ is the action of the geometric Frobenius at $\left(-1\right)^{n-1} c^{-1}$ acting on the stalk of $\wedge^m \mathcal{K}$.
\end{theorem}

\begin{proof}
	In the proof of \Cref{thm:L-function-of-exotic-kloosterman-sum}, we showed that for $0 \le m \le n$,
	\begin{align*}
		\left(\left(-1\right)^{r} q^{\frac{2n-m-1}{2}}\right)^m \besselFunction_{\pi, \fieldCharacter}\begin{pmatrix}
		& \IdentityMatrix{n-m}\\
		c \IdentityMatrix{m}
		\end{pmatrix}
		= \sum_{\mu=\left(m_1, \dots, m_t\right) \vdash m} \frac{1}{Z_\mu} \left(-1\right)^{nt} \prod_{j=1}^t  J_{m_j}\left(\alpha^{-1}, \fieldCharacter, \left(-1\right)^{n-1} c^{-1} \right).
	\end{align*}
	Therefore,
	\begin{align*}
	& q^{\frac{\left(n-m\right)m}{2}}  \besselFunction_{\pi, \fieldCharacter}\begin{pmatrix}
	& \IdentityMatrix{n - m}\\	
	c \IdentityMatrix{m} &
	\end{pmatrix}  \\
	&= \left(\left(-1\right)^{r-1} q^{ -\frac{\left(n-1\right)}{2} }\right)^m \sum_{\mu = \left(m_1,\dots,m_t\right) \vdash m} \frac{1}{Z_\mu} \left(-1\right)^{t+m} \prod_{j=1}^t  \left(-1\right)^{n-1}  J_{m_j}\left(\fieldCharacter, \alpha^{-1}, \left(-1\right)^{n-1} c^{-1}\right),
	\end{align*}
The theorem now follows from \Cref{thm:katz-etale-kloosterman-sheaf} (\ref{item:m-th-geometric-frobenius-trace-exotic-kloosterman-sum}) and from the following identity, which is true for any square complex valued matrix $A$:
$$ \trace\left( \wedge^m \left( A \right) \right) = \sum_{\mu = \left(m_1,\dots,m_t\right) \vdash m} \frac{1}{Z_\mu} \left(-1\right)^{m+t} \prod_{j=1}^t \trace\left(A^{m_j}\right).$$ 
\end{proof}

\begin{remark}
	\begin{enumerate}
		\item The proof of \Cref{thm:bessel-function-as-exterior-power} also shows that for $m > n$, $$ \left(\left(-1\right)^{r-1} q^{ -\frac{\left(n-1\right)}{2} }\right)^m \trace  \left( \Frobenius_{ \left(-1\right)^{n-1} c^{-1} } \mid \wedge^m \mathcal{K}_{ \left(-1\right)^{n-1} c^{-1} }  \right) = 0.$$
		We already know this due to \Cref{thm:katz-etale-kloosterman-sheaf} part (\ref{item:kloosterman-sheaf-is-of-rank-n}), but this is an independent proof based on gamma factors and representation theory.
		\item Recall that the eigenvalues of $\Frobenius_{ \left(-1\right)^{n-1} c^{-1} } \mid \mathcal{K}_{ \left(-1\right)^{n-1} c^{-1} }$ all have absolute value $q^{\frac{n-1}{2}}$, and therefore the eigenvalues of $\Frobenius_{ \left(-1\right)^{n-1} c^{-1} } \mid \wedge^m \mathcal{K}_{ \left(-1\right)^{n-1} c^{-1} }$ all have absolute value $q^{\frac{m \left(n-1\right)}{2}}$. Hence, the factor on the left hand side of \eqref{eq:bessel-function-as-exterior-power} is just a normalization factor.
		\item $c^{-1} \left(-1\right)^{n-1}$ is the determinant of the inverse of the matrix $\left( \begin{smallmatrix}
		& I_{n-1}\\
		c
		\end{smallmatrix} \right)$.
		\item If $m = 1$, we get the formula $$ \besselFunction_{\pi, \fieldCharacter}\begin{pmatrix}
		& \IdentityMatrix{n - 1}\\	
		c &
		\end{pmatrix} = \left(-1\right)^{n+r} q^{-n+1} \sum_{\substack{\xi \in \multiplicativegroup{\finiteFieldExtension{\lambda}} \\
				\prod_{i=1}^r \FieldNorm{n_i}{1}\left(\xi\right) = \left(-1\right)^{n-1} c^{-1}}} \alpha^{-1}\left(\xi\right) \fieldCharacter\left(\trace \xi \right).$$
		This formula is known in the literature for cuspidal representations \cite{tulunay2004cuspidal,Nien17} and for generic representations \cite[Lemma 3.5]{curtis2004zeta}.	
		\item Formulas similar to the one given in \Cref{thm:bessel-function-as-exterior-power} are known in the literature for several cases. For instance, a special case of the Shintani formula \cite{shintani1976explicit} gives the formula $$\trace \left(\wedge^m \left( A_\pi \right)\right) = q^{\frac{m \left(n-m\right)}{2}}  W^\circ \begin{pmatrix}
		\varpi I_m &\\
		& I_{n-m}
		\end{pmatrix},$$ where $W^\circ$ is the normalized spherical Whittaker function 
		of an unramified representation $\pi$ of $\GL_n\left(F\right)$, with Satake parameter $A_\pi$. Here $F$ is a non-archimedean local field with a uniformizer $\varpi$, and the Whittaker model is taken with respect to an additive character $\fieldCharacter$ with conductor $0$.
		\item 	The Bessel function satisfies the relation
		$$ \conjugate{\besselFunction_{\pi, \fieldCharacter}\begin{pmatrix}
			& \IdentityMatrix{n - m}\\	
			c \IdentityMatrix{m} &
			\end{pmatrix}} = \centralCharacter{\pi}\left(c\right)^{-1} \besselFunction_{\pi, \fieldCharacter}\begin{pmatrix}
		& \IdentityMatrix{m}\\	
		c \IdentityMatrix{n-m} &
		\end{pmatrix}.$$
		This identity reminds of the binomial identity $$ \binom{n}{m} = \binom{n}{n-m}.$$
		Both of these identities can be understood as a corollary of the isomorphism
		\begin{equation}
		\left(\wedge^m\left(V\right)\right)^{\vee} \cong \wedge^{n-m}\left(V\right) \otimes \left(\det V\right)^{-1},
		\end{equation}
		where $V$ is an $n$-dimensional complex representation of a group $G$.
	\end{enumerate}
\end{remark}

\begin{example}\label{example:steinberg-usual-kloosterman}
	Take $\lambda = \left(1,1,\dots,1\right) \vdash n$ and $\alpha = 1$ the trivial character. Then $\pi = \Pi_{\left(1,1\dots,1\right)}\left(1\right)$ is the unique irreducible generic subrepresentation of $1 \circ 1 \circ \dots \circ 1$. That is, $\pi$ is the Steinberg representation $\pi = \Steinberg_n$ of $\GL_n\left(\finiteField\right)$. In this case, $$J_{m}\left(\alpha^{-1}, \fieldCharacter, a\right) = \mathrm{Kl}_m\left(\fieldCharacter, a\right) = \sum_{\substack{x_1,\dots,x_n \in \multiplicativegroup{\finiteFieldExtension{m}}\\
	\prod_{i=1}^n x_i = a}} \fieldCharacter\left(\sum_{i=1}^n x_i \right)$$ is the usual Kloosterman sum, and $$\mathcal{K} = \mathrm{Kl}_n\left(\fieldCharacter\right) = \convolutionWithCompactSupport \mathrm{prod}_{!} \left( \mathrm{add}^{\ast} \artinScrier \right) \left[n - 1\right]$$ is the usual Kloosterman sheaf associated with the diagram 
\begin{equation}\label{eq:simple-kloosterman-diagram}
	\xymatrix{ & \multiplcativeScheme \ar[ld]_{ \mathrm{prod} }  \ar[rd]^{ \mathrm{add} } \\
		\multiplcativeScheme & & \affineLine },
\end{equation}
see also \cite[Section 7.4.2]{katz2016gauss} and  \cite[Section 0.1]{heinloth2013kloosterman}.

By \Cref{thm:L-function-of-exotic-kloosterman-sum}, we have that $$L^{\ast}\left(T, \mathrm{Kl}\left(\fieldCharacter, \left(-1\right)^{n-1} c^{-1} \right)\right) = \sum_{m =0 }^{n} q^{\frac{m \left(n-m\right)}{2}}  \besselFunction_{\Steinberg_n, \fieldCharacter} \begin{pmatrix}
& \IdentityMatrix{n-m}\\
c \IdentityMatrix{m}
\end{pmatrix} T^m,$$
and by \Cref{thm:bessel-function-as-exterior-power}, we have that $$\left(\left(-1\right)^{n-1} q^{ -\frac{\left(n-1\right)}{2} }\right)^m \trace  \left( \Frobenius_{ \left(-1\right)^{n-1} c^{-1} } \mid \wedge^m \mathcal{K}_{ \left(-1\right)^{n-1} c^{-1} }  \right) = q^{\frac{m \left(n-m\right)}{2}} \besselFunction_{\Steinberg_n, \fieldCharacter}\begin{pmatrix}
& \IdentityMatrix{n - m}\\	
c \IdentityMatrix{m} &
\end{pmatrix}.$$
This shows a deep relation between the Steinberg representation and usual Kloosterman sums and sheaves.
\end{example}
\begin{example}\label{example:principal-series-twisted-kloosterman}
	Take $\lambda = \left(1,1,\dots,1\right) \vdash n$ as in the previous example, and let $\alpha = \left(\alpha_1, \dots, \alpha_n\right)$, where $\alpha_i \colon \multiplicativegroup{\finiteField} \rightarrow \multiplicativegroup{\cComplex}$ is a multiplicative character for every $i$. Then $\pi = \Pi_{\left(1,1\dots,1\right)}\left(\alpha_1,\dots, \alpha_n\right)$ is the unique irreducible generic subrepresentation of $\alpha_1 \circ \dots \circ \alpha_n$, that is, $\pi$ is a generic principal series representation of $\GL_n\left(\finiteField\right)$. In this case, we have that $$J_m\left(\alpha^{-1}, \psi, a\right) = \sum_{\substack{x_1,\dots,x_n \in \multiplicativegroup{\finiteFieldExtension{m}}\\
	\prod_{i=1}^n x_i = a}} \left(\prod_{i=1}^n \alpha_i^{-1} \left(\FieldNorm{m}{1}\left(x_i\right)\right)\right) \fieldCharacter\left(\sum_{i=1}^n x_i\right) $$ is the twisted Kloosterman sum, and $$\mathcal{K} = \mathrm{Kl}_n\left(\alpha^{-1},\fieldCharacter\right) = \convolutionWithCompactSupport \mathrm{prod}_{!} \left( \mathrm{add}^{\ast} \artinScrier \otimes \mathcal{L}_{\alpha^{-1}}  \right) \left[n - 1\right]$$
is the twisted Kloosterman sheaf, associated with diagram \eqref{eq:simple-kloosterman-diagram}, see \cite[Section 4.1.1]{katz2016gauss} and \cite[Section 6]{kowalski2018stratification}.

Hence, we obtain a deep relation between Bessel functions of generic principal series representations and twisted Kloosterman sums and sheaves.
\end{example}

\subsection{Polynomials with unitary roots}

As an application of the relation we established in \Cref{thm:bessel-function-as-exterior-power} between values of the Bessel function and the exotic Kloosterman sheaves, we will show that some polynomials closely related to the normalized $L$-function associated with an exotic Kloosterman sum have all of their roots lying on the unit circle. In one of the cases consider, this statement is actually the Riemann hypothesis for the $L$-function we defined. We are not aware of any direct representation theoretic proof of this purely representation theoretic statement.

\begin{theorem}
	Let $\pi$ be an irreducible generic representation of $\GL_n\left(\finiteField\right)$. Then the following polynomials have all of their roots lying on the unit circle. \begin{align*}
		P\left(X\right) &= \sum_{m=0}^n \besselFunction_{\pi, \fieldCharacter}\begin{pmatrix}
		& \IdentityMatrix{m}\\	
		c \IdentityMatrix{n-m} &
		\end{pmatrix} q^{\frac{m\left(n-m\right)}{2}} X^m,\\
		Q\left(X\right) &= \sum_{m=0}^n \besselFunction_{\pi, \fieldCharacter}\begin{pmatrix}
		& \IdentityMatrix{m}\\	
		c \IdentityMatrix{n - m} &
		\end{pmatrix} X^m.
	\end{align*}
\end{theorem}
\begin{proof}
	If $n=1$, then the statement is trivial.
	
	Suppose $n > 1$. Write $\pi = \Pi_{\lambda}\left(\alpha\right)$, where $\lambda \vdash n$ and $\alpha \in \charactergroup{\lambda}$. Let $\mathcal{K} = \operatorname{Kl}\left(\finiteFieldExtension{\lambda}, \alpha^{-1}, \fieldCharacter\right)$. Let $A = \Frobenius_{ \left(-1\right)^{n-1} c^{-1} } \mid \mathcal{K}_{ \left(-1\right)^{n-1} c^{-1} }$, as in \Cref{thm:katz-etale-kloosterman-sheaf}. The characteristic polynomial of $A$ is given by
	$$ \operatorname{Char}_A\left(X\right) = \sum_{m=0}^n { \left(-1\right)^{n - m} \trace \left(\wedge^{n-m}\left(A\right)\right) X^m }.$$
	By \Cref{thm:katz-etale-kloosterman-sheaf} (\ref{item:kloosterman-is-pure-of-weight-n-minus-1}), we have that the roots of $\operatorname{Char}_A\left(X\right)$ have absolute value $q^{\frac{n-1}{2}}$. By \Cref{thm:bessel-function-as-exterior-power}, we have that
	$$ \operatorname{Char}_A\left(X\right) = q^{\frac{n \left(n-1\right)}{2}} \sum_{m=0}^{n} \left( \left(-1\right)^{r} q^{\frac{m}{2}} \right)^{n-m} \besselFunction_{\pi, \fieldCharacter}\begin{pmatrix}
	& \IdentityMatrix{m}\\	
	c \IdentityMatrix{n-m} &
	\end{pmatrix} \left(q^{-\frac{n-1}{2}} X\right)^m.$$
	Hence, we have that $P \left(X\right) = q^{-\frac{n \left(n-1\right)}{2}} \left(-1\right)^{rn} \operatorname{Char}_A(\left(-1\right)^r q^{\frac{n-1}{2}} X)$, so the roots of $P\left(X\right)$ are all of absolute value $1$.
	
	For the result for $Q\left(X\right)$, we use \Cref{prop:unitary-polynomial-deformation-is-unitary} by noticing that $Q\left(X\right) = P_{q^{-\frac{1}{2}}}\left(X\right)$.
\end{proof}

As a corollary, we obtain the following non-trivial bounds on special values of the Bessel function.
\begin{corollary}For any irreducible generic representation $\pi$ of $\GL_n\left(\finiteField\right)$ and any $c \in \multiplicativegroup{\finiteField}$ and any $0 \le m \le n$,
	$$\abs{\besselFunction_{\pi, \fieldCharacter} \begin{pmatrix}
		& \IdentityMatrix{n-m}\\
		c \IdentityMatrix{m}
		\end{pmatrix}} \le \binom{n}{m} q^{-\frac{m \left(n-m\right)}{2}}.$$
\end{corollary}	

\subsection{Shintani base change and Dickson polynomials}

We show another application of our results, relating special values of the Bessel function of an irreducible generic representation to special values of the Bessel function of its Shintani base change, using Dickson polynomials.

Let $k \ge 1$, and let $\fieldCharacter_k = \fieldCharacter \circ \FieldTrace_{\FieldExtension{\finiteFieldExtension{k}}{\finiteField}} \colon \finiteFieldExtension{k} \rightarrow \multiplicativegroup{\cComplex}$.

Shintani \cite{shintani1976two} defined a descent map from irreducible representations $\pi_k$ of $\GL_n\left( \finiteFieldExtension{k} \right)$ that satisfy $\pi_k \circ \Frobenius \cong \pi_k$ to irreducible representations of $\GL_n \left(\finiteField\right)$. Silberger and Zink described the inverse of this map, called Shintani base change, in terms of Macdonald's parametrization \cite[Corollary 9.6]{SilbergerZink08}.

In order to describe this map, we first notice that the groups $\limitcharactergroup = \varinjlim \charactergroup{d}$ and $\limitcharactergroup^{\finiteFieldExtension{k}} = \varinjlim \charactergroup{kd}$ are isomorphic in a natural way.

Suppose that $\pi$ is an irreducible representation of $\GL_n\left(\finiteField\right)$ parameterized by $\phi \in P_n\left(\Gamma\right)$. Then its Shintani base change $\pi_k$, an irreducible representation of $\GL_n\left(\finiteFieldExtension{k}\right)$, is the representation corresponding to the parameter $\phi_k \colon \limitcharactergroup^{\finiteFieldExtension{k}} \rightarrow \partitionset$, such that $\phi_k\left(\gamma\right) = \phi\left(\gamma\right)$, for any $\gamma \in \limitcharactergroup^{\finiteFieldExtension{k}} \cong \limitcharactergroup$.

It follows from our description of irreducible generic representations of $\GL_n\left(\finiteField\right)$ that the Shintani base change of an irreducible generic representation of $\GL_n\left(\finiteField\right)$ is an irreducible generic representation of $\GL_n\left(\finiteFieldExtension{k}\right)$.

If $\pi$ is an irreducible cuspidal representation of $\GL_n\left(\finiteField\right)$ corresponding to a regular character $\alpha \colon \multiplicativegroup{\finiteFieldExtension{n}} \rightarrow \multiplicativegroup{\cComplex}$, then its Shintani base change $\pi_k$ is the parabolic induction $$\Pi^{\finiteFieldExtension{k}}_{f_1} \circ \dots \circ \Pi^{\finiteFieldExtension{k}}_{f_{\gcd(n, k)}},$$ where $f_i$ is the $\finiteFieldExtension{k}$-Frobenius orbit corresponding to $\alpha^{q^{i-1}} \circ \FieldNorm{\lcm(n, k)}{n} \colon \multiplicativegroup{\finiteFieldExtension{\lcm(n,k)}} \rightarrow \multiplicativegroup{\cComplex} $ and $\Pi_{f_i}^{\finiteFieldExtension{k}}$ is the irreducible cuspidal representation of $\GL_{{\lcm (n,k)} / {k}}\left( \finiteFieldExtension{k} \right)$ corresponding to $f_i$.

We therefore get that for $\lambda = \left(n_1, \dots, n_r\right) \vdash n$, and $\alpha = \left(\alpha_1, \dots, \alpha_r\right) \in \charactergroup{\lambda}$, the Shintani base change of $\pi = \Pi_\lambda\left(\alpha_1, \dots, \alpha_r\right)$ is $\pi_k = \Pi^{\finiteFieldExtension{k}}_{\lambda'}\left(\alpha'\right)$, where $\lambda'$ is the concatenation of $$\left(\left(\frac{n_j}{\gcd\left(n_j, k\right)}, \dots, \frac{n_j}{\gcd\left(n_j, k\right)} \right)\right)_{j=1}^r$$ and $\alpha'$ is the concatenation of
$$\left(\left(\alpha_j \circ \FieldNorm{\lcm(n_j, k)}{n_j}, \alpha_1^q \circ \FieldNorm{\lcm(n_j, k)}{n_j}, \dots, \alpha_j^{q^{\gcd(n_j,k)-1}} \circ \FieldNorm{\lcm(n_j, k)}{n_j} \right)\right)_{j=1}^r.$$
We have $\lengthof\left(\lambda'\right) = \sum_{j=1}^r \gcd\left(n_j,k\right)$ and \begin{equation}\label{eq:sign-of-length-of-base-change}
	\left(-1\right)^{\lengthof\left(\lambda'\right)} = \prod_{j=1}^r \left(-1\right)^{\gcd(n_j, k)} = \prod_{j=1}^r \left(-1\right)^{n_j + k + n_j k } = \left(-1\right)^{n + kr + nk}.
\end{equation}

We begin by comparing the exotic Kloosterman sums associated to $\alpha^{-1}$ and $\alpha'^{-1}$ over $\finiteField$ and $\finiteFieldExtension{k}$, respectively.
\begin{proposition}\label{prop:base-change-exotic-kloosterman-formula}
	Consider the exotic 
	Kloosterman sum $$J_m^{\finiteFieldExtension{k}}\left(\alpha'^{-1}, \fieldCharacter_k, a\right) = \sum_{\substack{x \in \finiteFieldExtension{\lambda'} \otimes_{\finiteFieldExtension{k}} \finiteFieldExtension{km}\\
			\etaleNorm_{2}^{\finiteFieldExtension{\lambda'}\otimes_{\finiteFieldExtension{k}} \finiteFieldExtension{mk}}\left(x\right) = a}} \alpha'^{-1}\left( \etaleNorm_{1}^{\finiteFieldExtension{\lambda'} \otimes_{\finiteFieldExtension{k}} \finiteFieldExtension{km}} \left(x\right) \right) \fieldCharacter_k\left(\trace^{\finiteFieldExtension{k}} x \right),$$
	where $a \in \multiplicativegroup{\finiteField}$ and $\trace^{\finiteFieldExtension{k}}$ is the trace to $\finiteFieldExtension{k}$.
	Then $$J_m^{\finiteFieldExtension{k}}\left(\alpha'^{-1}, \fieldCharacter_k, a\right) = J_{km}\left(\alpha^{-1}, \fieldCharacter, a\right).$$
\end{proposition}

\begin{proof}
	As in \Cref{sec:tensor-product-of-finite-fields}, we have an isomorphism of $\finiteField$-algebras \begin{equation}\label{eq:isomorphism-of-base-change-with-tensor-product}
	T \colon \finiteFieldExtension{\lambda'} = \prod_{j=1}^r \prod_{i = 1}^{\gcd(n_j, k)} \finiteFieldExtension{\lcm\left( n_j,k \right)} \rightarrow \prod_{j=1}^r \finiteFieldExtension{n_j} \otimes_\finiteField \finiteFieldExtension{k} = \finiteFieldExtension{\lambda} \otimes_\finiteField \finiteFieldExtension{k},
	\end{equation} such that $\alpha' = \alpha \circ \etaleNorm_{1}^{\finiteFieldExtension{\lambda} \otimes_{\finiteField} \finiteFieldExtension{k}} \circ T.$
	
	Consider the tensor product $\finiteFieldExtension{\lambda'} \otimes_{\finiteFieldExtension{k}} \finiteFieldExtension{km}$. We get from \eqref{eq:isomorphism-of-base-change-with-tensor-product} an isomorphism of $\finiteField$-algebras $$T \otimes \idmap_{\finiteFieldExtension{km}} \colon \finiteFieldExtension{\lambda'} \otimes_{\finiteFieldExtension{k}} \finiteFieldExtension{km} \rightarrow \left(\finiteFieldExtension{\lambda} \otimes_{\finiteField} \finiteFieldExtension{k}\right) \otimes_{\finiteFieldExtension{k}} \finiteFieldExtension{km} 
	,$$
	and under this isomorphism, we have \begin{align*}
	T \circ \etaleNorm_{1}^{\finiteFieldExtension{\lambda'} \otimes_{\finiteFieldExtension{k}} \finiteFieldExtension{km}} &= \etaleNorm_{1}^{\left(\finiteFieldExtension{\lambda} \otimes_\finiteField \finiteFieldExtension{k}\right) \otimes_{\finiteFieldExtension{k}} \finiteFieldExtension{km}} \circ T \otimes \idmap_{\finiteFieldExtension{km}},\\
	\etaleNorm_{2}^{\finiteFieldExtension{\lambda'} \otimes_{\finiteFieldExtension{k}} \finiteFieldExtension{km}} &= \etaleNorm_{2}^{\left(\finiteFieldExtension{\lambda} \otimes_\finiteField \finiteFieldExtension{k}\right) \otimes_{\finiteFieldExtension{k}} \finiteFieldExtension{km}} \circ T \otimes \idmap_{\finiteFieldExtension{km}}.
	\end{align*}
	
	Using the isomorphism $T$ above, we may rewrite the exotic Kloosterman sum as
	
	$$J_m^{\finiteFieldExtension{k}}\left(\alpha'^{-1}, \fieldCharacter, a\right) = \sum_{\substack{x \in \left(\finiteFieldExtension{\lambda} \otimes_{\finiteField} \finiteFieldExtension{k}\right) \otimes_{\finiteFieldExtension{k}} \finiteFieldExtension{km}\\
			\etaleNorm_{2}^{\left(\finiteFieldExtension{\lambda} \otimes_{\finiteField} \finiteFieldExtension{k}\right) \otimes_{\finiteFieldExtension{k}} \finiteFieldExtension{km}}\left(x\right) = a}} \alpha^{-1}\left( \etaleNorm_{1}^{\finiteFieldExtension{\lambda} \otimes_{\finiteField} \finiteFieldExtension{k}} \left(\etaleNorm_{1}^{\left(\finiteFieldExtension{\lambda} \otimes_{\finiteField} \finiteFieldExtension{k}\right) \otimes_{\finiteFieldExtension{k}} \finiteFieldExtension{km}} \left(x\right)\right) \right) \fieldCharacter_k \left(\trace^{\finiteFieldExtension{k}} x \right).$$
	We have a natural isomorphism of $\finiteField$-algebras $$S \colon \left(\finiteFieldExtension{\lambda} \otimes_{\finiteField} \finiteFieldExtension{k}\right) \otimes_{\finiteFieldExtension{k}} \finiteFieldExtension{km} \rightarrow \finiteFieldExtension{\lambda} \otimes_{\finiteField} \finiteFieldExtension{km}.$$
	Under this isomorphism, we have
	\begin{align*}
	\etaleNorm_1^{\finiteFieldExtension{\lambda} \otimes_{\finiteField} \finiteFieldExtension{k}} \circ  \etaleNorm_{1}^{\left(\finiteFieldExtension{\lambda} \otimes_{\finiteField} \finiteFieldExtension{k}\right) \otimes_{\finiteFieldExtension{k}} \finiteFieldExtension{km} } &= \etaleNorm_1^{\finiteFieldExtension{\lambda} \otimes_{\finiteField} \finiteFieldExtension{km}} \circ S,\\
	\etaleNorm_{2}^{\left(\finiteFieldExtension{\lambda} \otimes_{\finiteField} \finiteFieldExtension{k}\right) \otimes_{\finiteFieldExtension{k}} \finiteFieldExtension{km}} &= \etaleNorm_2^{\finiteFieldExtension{\lambda} \otimes_{\finiteField} \finiteFieldExtension{km}} \circ S,\\
	\fieldCharacter_k \circ \trace^{\finiteFieldExtension{k}} &= \fieldCharacter \circ \trace \circ S.
	\end{align*}
	
	Hence, we get that $$J_m^{\finiteFieldExtension{k}}\left(\alpha'^{-1}, \fieldCharacter_k, a\right) = J_{km}\left(\alpha^{-1}, \fieldCharacter, a\right),$$
as required.
\end{proof}

Consider the $L$-functions associated to the exotic Kloosterman sums $J\left(\alpha^{-1}, \fieldCharacter, a\right)$ and $J^{\finiteFieldExtension{k}}\left(\alpha'^{-1}, \fieldCharacter, a\right)$ for $a \in \multiplicativegroup{\finiteField}$. Write
\begin{align*}
	L\left(T, J\left(\alpha^{-1}, \fieldCharacter, a\right)\right)^{\left(-1\right)^n} &= \prod_{j=1}^n \left(1 - \omega_j T\right),\\
	L\left(T, J^{\finiteFieldExtension{k}}\left(\alpha'^{-1}, \fieldCharacter, a\right)\right)^{\left(-1\right)^n} &= \prod_{j=1}^n \left(1 - \omega_j' T\right).
\end{align*}
By \Cref{cor:hasse-davenport-for-exotic-kloosterman-sums} and \Cref{prop:base-change-exotic-kloosterman-formula}, we have that for every $m$, $$\sum_{j=1}^n \omega_j^{km} = \sum_{j=1}^n {\left(\omega'_j\right)}^{m},$$
which implies that without loss of generality we have $\omega'_j = \omega_j^k$ for every $j$, and therefore \begin{align*}
L\left(T, J^{\finiteFieldExtension{k}}\left(\alpha'^{-1}, \fieldCharacter, a\right)\right)^{\left(-1\right)^n} &= \prod_{j=1}^n \left(1 - \omega_j^k T\right).
\end{align*}
This implies that we have a relation for the normalized $L$-functions.
\begin{equation}\label{eq:relation-between-normalized-L-functions}
	\begin{split}
		L^{\ast}\left(T, J\left(\alpha^{-1}, \fieldCharacter, a\right)\right) &= \prod_{j=1}^n \left(1 -  q^{-\frac{(n-1)}{2}} \left(-1\right)^r \omega_j T\right),\\
		L^{\ast}\left(T, 
		J^{\finiteFieldExtension{k}}\left(\alpha'^{-1}, \fieldCharacter, a\right)
		\right) &= \prod_{j=1}^n \left(1 -  q^{-\frac{k(n-1)}{2}} \left(-1\right)^{kr} \omega_j^k \left(-1\right)^{\left(k-1\right)n} T\right),
	\end{split}
\end{equation}
where we used \eqref{eq:sign-of-length-of-base-change} for the second formula.

The relation between the coefficients of these normalized $L$-functions is given by Dickson polynomials. We explain this now. Our reference is \cite[Page 17, Section 2.4]{lidl1993longman}.

Let $T_1, \dots, T_n$ be $n$ formal variables. Consider the following polynomial over $\zIntegers$ \begin{equation}\label{eq:generic-polynomial-with-roots}
	\prod_{i=1}^n \left(X - T_i\right) = \sum_{j=0}^n \left(-1\right)^{n-j} S_{n-j}\left(T_1,\dots,T_n\right) X^j,
\end{equation} where $S_0\left(T_1,\dots,T_n\right) = 1$ and 
$$S_j\left(T_1,\dots,T_n\right) = \sum_{1 \le  i_1 < \dots < i_j \le n} T_{i_1} \dots T_{i_j}.$$ The polynomials $\left(S_j\right)_{j=1}^n$ are called the elementary symmetric polynomials. They serve as generators for the ring of symmetric polynomials in $\zIntegers$ coefficients. Consider the polynomial \begin{equation}
	\label{eq:generic-polynomial-with-roots-to-the-k} \prod_{i=1}^n \left(X - T_i^k\right) = \sum_{j=0}^n \left(-1\right)^{n-j} S_{n-j}\left(T_1^k,\dots,T_n^k\right) X^j.
\end{equation}
Its coefficients are symmetric polynomials in $\left(T_1,\dots,T_n\right)$ with $\zIntegers$ coefficients, and therefore can expressed as polynomials with $\zIntegers$ coefficients in the elementary symmetric polynomials $\left(S_j\right)_{j=1}^n$. The multivariate Dickson polynomials in $n$ variables allow us to express the coefficients of \eqref{eq:generic-polynomial-with-roots-to-the-k} in terms of the coefficients of \eqref{eq:generic-polynomial-with-roots}.
The polynomials $(D_j^{(k)})_{j=1}^n$ are defined as the unique polynomials such that
$$\prod_{i=1}^n \left(X - T_i^k\right) = \sum_{j=0}^n \left(-1\right)^{n-j} D^{(k)}_{n-j}\left(S_1\left(T_1,\dots,T_n\right),\dots,S_n\left(T_1,\dots,T_n\right)\right) X^j.$$ Equivalently,  $D^{(k)}_j\left(U_1,\dots,U_n\right)$ is the unique polynomial such that
$$S_j\left(T_1^k,\dots,T_n^k\right) = D_{j}^{(k)}\left(S_1\left(T_1,\dots,T_n\right),\dots,S_n\left(T_1,\dots,T_n\right)\right).$$
See also \cite[Page 18, Definition 2.21]{lidl1993longman}.

The relation \eqref{eq:relation-between-normalized-L-functions} we found between the normalized $L$-functions can be translated using Dickson polynomials to the following purely representation theoretic statement.

\begin{theorem}
	Let $\pi$ be an irreducible generic representation of $\GL_n\left(\finiteField\right)$. Let $\pi_k$ be the Shintani base change of $\pi$ to $\GL_n \left(\finiteFieldExtension{k}\right)$. Then for any $0 \le m \le n$, and any $c \in \multiplicativegroup{\finiteField}$,
	$$D^{(k)}_m \left(b_1, \dots, b_n\right) =  \left(\left(-1\right)^{k-1} q^{ \frac{k}{2}}\right)^{m \left(n - m\right)} \besselFunction_{\pi_k, \fieldCharacter_k} \begin{pmatrix}
	& \IdentityMatrix{n-m}\\
	c\IdentityMatrix{m}
	\end{pmatrix},$$
	where for every $1 \le  j \le n$, $$b_j = q^{\frac{j\left(n-j\right)}{2}} \besselFunction_{\pi,\fieldCharacter}\begin{pmatrix}
	& \IdentityMatrix{n-j}\\
	c \IdentityMatrix{j}
	\end{pmatrix}.$$
\end{theorem}
\begin{proof}
	Let $a = \left(-1\right)^{n-1}c^{-1}$. In order to make use of the definition of Dickson polynomials, we rewrite \eqref{eq:relation-between-normalized-L-functions} using monic polynomials. Using the same notations as before and using \Cref{thm:L-function-of-exotic-kloosterman-sum}, we have that
\begin{equation*}
	\begin{split}
		 \prod_{j=1}^n \left(T - q^{-\frac{(n-1)}{2}}(-1)^{r-1} \omega_j\right) &= T^n L^{\ast}\left(-T^{-1}, J\left(\alpha^{-1}, \fieldCharacter, a\right)\right) \\
		 &= \sum_{m = 0}^n \left(-1\right)^{n-m} q^{\frac{m \left(n-m\right)}{2}}  \besselFunction_{\pi, \fieldCharacter} \begin{pmatrix}
			& \IdentityMatrix{m}\\
			c \IdentityMatrix{n-m}
		\end{pmatrix} T^m,
	\end{split}
\end{equation*}
and that
\begin{align*}
	& \prod_{j=1}^n \left(T - q^{-\frac{k(n-1)}{2}}(-1)^{k(r-1)} \omega_j^k\right) \\
	= & \,T^n L^{\ast}\left(\left(-1\right)^{\left(k-1\right)\left(n-1\right)} \left(-T^{-1}\right), 
	J^{\finiteFieldExtension{k}}\left(\alpha'^{-1}, \fieldCharacter, a\right)
	\right) \\
	=& \, \sum_{m = 0}^n \left(-1\right)^{n-m} \left(-1\right)^{m\left(k-1\right)\left(n-1\right)} q^{\frac{km \left(n-m\right)}{2}}  \besselFunction_{\pi_k, \fieldCharacter_k} \begin{pmatrix}
	& \IdentityMatrix{m}\\
	c \IdentityMatrix{n-m}
	\end{pmatrix} T^m.
\end{align*}
The result now follows immediately from the definition of Dickson polynomials and from the fact that $\left(-1\right)^{m\left(n-1\right)}=\left(-1\right)^{m\left(n-m\right)}$.

\end{proof}

\subsection{Converse theorem for Kloosterman sheaves}
In this final section, we combine the results of \Cref{sec:tensor-product-of-finite-fields} and \Cref{subsec:katz-exotic-kloosterman-sheaves} with Nien's converse theorem \cite[Theorem 3.9]{Nien14} (or more precisely, its improved version in \cite[Theorem 4.8]{SoudryZelingher2023}) to deduce a converse theorem for exotic Kloosterman sheaves.

We begin with recalling the (improved) converse theorem.

\begin{theorem}\label{thm:improved-converse-theorem}
	Let $\pi_1$ and $\pi_2$ be irreducible generic representations of $\GL_n\left(\finiteField\right)$ with the same central character. Suppose that for any $1 \le m \le \frac{n}{2}$ and that for any irreducible cuspidal representation $\sigma$ of $\GL_m\left(\finiteField\right)$ we have that $$\gamma\left( \pi_1 \times \sigma, \fieldCharacter \right) = \gamma\left( \pi_2 \times \sigma, \fieldCharacter \right).$$
	Then $\pi_1 \cong \pi_2$.
\end{theorem}

We will now translate this to a statement about Kloosterman sheaves.

Let $\lambda = \left(n_1,\dots,n_r\right)$ and $\lambda' = \left(n'_1,\dots,n'_{r'}\right)$ be partitions of $n$. For any $1 \le j \le r$ (respectively, $1 \le j \le r'$) let $\alpha_j \colon \multiplicativegroup{\finiteFieldExtension{n_j}} \to \multiplicativegroup{\cComplex}$ (respectively, $\alpha'_{j} \colon \multiplicativegroup{\finiteFieldExtension{n'_{j}}} \to \multiplicativegroup{\cComplex}$) be a regular multiplicative character. Denote $\alpha = \left(\alpha_1,\dots,\alpha_r\right) \colon \multiplicativegroup{\finiteFieldExtension{\lambda}} \to \multiplicativegroup{\cComplex}$ and $\alpha' = \left(\alpha'_1,\dots,\alpha'_{r'}\right) \colon \multiplicativegroup{\finiteFieldExtension{\lambda'}} \to \multiplicativegroup{\cComplex}$. For any $m$ we denote by $\mathcal{K}_{\lambda}^m\left(\alpha\right)$ (respectively, $\mathcal{K}_{\lambda'}^m\left(\alpha'\right)$) the base change of $\operatorname{Kl}\left( \finiteFieldExtension{\lambda}, \alpha, \fieldCharacter \right)$ (respectively, of $\operatorname{Kl}\left( \finiteFieldExtension{\lambda'}, \alpha', \fieldCharacter \right)$) from $\finiteField$ to $\finiteFieldExtension{m}$.

\begin{theorem}
	 Suppose that $\alpha\restriction_{\multiplicativegroup{\finiteField}} = \alpha'\restriction_{\multiplicativegroup{\finiteField}}$ and that for any $1 \le m \le \frac{n}{2}$ and any  $\finiteFieldExtension{m}$-point $a$ of $\multiplcativeScheme$ we have that \begin{equation}\label{eq:kloosterman-converse-theorem-assumption}
	 	\left(-1\right)^{m r} \trace\left(\Frobenius_a\mid \mathcal{K}^{m}_{\lambda}\left(\alpha\right)_a \right) = \left(-1\right)^{m r'} \trace\left(\Frobenius_a\mid \mathcal{K}^{m}_{\lambda'}\left(\alpha'\right)_a \right).
	 \end{equation}
	 Then $\lambda = \lambda'$ and, up to reordering, for any $1 \le j \le r$, the characters $\alpha_j$ and $\alpha'_j$ lie in the same Frobenius orbit.
\end{theorem}
\begin{proof}
	Let $\pi_1 = \Pi_{\lambda}\left(\alpha^{-1}\right)$ and $\pi_2 = \Pi_{\lambda'}\left(\left(\alpha'\right)^{-1}\right)$. Then $\pi_1$ and $\pi_2$ are irreducible generic representations of $\GL_n\left(\finiteField\right)$ that have the same central character $\alpha^{-1}\restriction_{\multiplicativegroup{\finiteField}} = \left(\alpha'\right)^{-1}\restriction_{\multiplicativegroup{\finiteField}}$. We will show that for any $1 \le m \le \frac{n}{2}$ and any irreducible cuspidal representation $\sigma$ of $\GL_m\left(\finiteField\right)$ we have that $$\gamma\left(\pi_1 \times \sigma, \fieldCharacter\right) = \gamma\left(\pi_2 \times \sigma, \fieldCharacter\right).$$
	
	By \eqref{eq:kloosterman-converse-theorem-assumption} and by \Cref{thm:katz-etale-kloosterman-sheaf} part \ref{item:m-th-geometric-frobenius-trace-exotic-kloosterman-sum}, we have that for any $a \in \multiplicativegroup{\finiteFieldExtension{m}},$
	\begin{equation}\label{eq:kloosterman-converse-base-change-equality}
		\left(-1\right)^{mr} \sum_{\substack{x \in \finiteFieldExtension{\lambda} \otimes_\finiteField \finiteFieldExtension{m} \\
				\etaleNorm^{\finiteFieldExtension{\lambda} \otimes_\finiteField \finiteFieldExtension{m}}_2(x) = a}} \alpha ( \etaleNorm^{\finiteFieldExtension{\lambda}\otimes_\finiteField \finiteFieldExtension{m}}_1 \left( x \right) ) \fieldCharacter\left( \trace x \right) = \left(-1\right)^{mr'} \sum_{\substack{x \in \finiteFieldExtension{\lambda'} \otimes_\finiteField \finiteFieldExtension{m} \\
				\etaleNorm^{\finiteFieldExtension{\lambda'} \otimes_\finiteField \finiteFieldExtension{m}}_2(x) = a}} \alpha' ( \etaleNorm^{\finiteFieldExtension{\lambda'}\otimes_\finiteField \finiteFieldExtension{m}}_1 \left( x \right) ) \fieldCharacter\left( \trace x \right).
	\end{equation}
		Let $1 \le m \le \frac{n}{2}$ and let $\sigma$ be an irreducible cuspidal representation of $\GL_m\left(\finiteField\right)$ corresponding to the Frobenius orbit of a regular character $\beta \colon \multiplicativegroup{\finiteFieldExtension{m}} \to \multiplicativegroup{\cComplex}$. By multiplying \eqref{eq:kloosterman-converse-base-change-equality} by $\left(-1\right)^{nm + n} \beta^{-1}\left(a\right)$ and summing over all $a \in \multiplicativegroup{\finiteFieldExtension{m}}$, we get
		$$\tau_{\lambda, m}\left(\alpha^{-1} \times \beta, \fieldCharacter\right) = \tau_{\lambda', m}\left(\left(\alpha'\right)^{-1} \times \beta, \fieldCharacter\right).$$
		By \Cref{prop:epsilon-factors-etale-product} this implies that $$\epsilon_0\left(\pi_1 \times \sigma, \fieldCharacter\right) = \epsilon_0\left(\pi_2 \times \sigma, \fieldCharacter\right),$$
		and by \Cref{thm:relation-between-epsilon-factors-and-gamma-factors} we have that $$ \gamma\left(\pi_1 \times \sigma, \fieldCharacter\right) = \gamma\left(\pi_2 \times \sigma, \fieldCharacter\right).$$
		
		Therefore, we have that the assumptions of \Cref{thm:improved-converse-theorem} hold. By \Cref{thm:improved-converse-theorem}, we get that $\pi_1 \cong \pi_2$, and therefore, $\lambda = \lambda'$ and, up to reordering, for any $1 \le j \le r$, the characters $\alpha_j$ and $\alpha'_j$ lie in the same Frobenius orbit.
\end{proof}
\begin{remark}
	In \cite[Remark B.3]{nien2021converse} it is shown that if $n < \frac{q-1}{2 \sqrt{q}} + 1$ then it suffices to check \eqref{eq:kloosterman-converse-theorem-assumption} for $m = 1$ (instead of for every $1 \le m \le \frac{n}{2}$). 
\end{remark}

\appendix

\section{Lemmas about Gauss sums}

In this appendix, we prove several lemmas about Gauss sums.

\begin{proposition}\label{prod:tensor-product-gauss-sum-invariant-under-gcd}
	Let $\alpha \in \charactergroup{n}$ and $\beta \in \charactergroup{m}$. Let $d = \gcd\left(n,m\right)$ and $l = \lcm\left(n,m\right)$. Then for any $i$, 
	$$ \tau\left( \alpha \circ \FieldNorm{l}{n} \cdot \beta^{q^{i+d}} \circ \FieldNorm{l}{m}, \fieldCharacter_l \right) = \tau\left( \alpha \circ \FieldNorm{l}{n} \cdot \beta^{q^{i}} \circ \FieldNorm{l}{m}, \fieldCharacter_l \right).$$
\end{proposition}
\begin{proof}
	Recall that Gauss sums are constant on Frobenius orbits. Therefore, for any $r \in \zIntegers$ we have
	$$ \tau\left( \alpha \circ \FieldNorm{l}{n} \cdot \beta^{q^{i}} \circ \FieldNorm{l}{m}, \fieldCharacter_l \right) = \tau\left( \alpha \circ \FieldNorm{l}{n}^{q^r} \cdot \beta^{q^{i}} \circ \FieldNorm{l}{m}^{q^r}, \fieldCharacter_l \right).$$
	There exist $a,b\in \zIntegers$, such that $an + bm = d$. Then
	 \begin{align*}
	 \tau\left( \alpha \circ \FieldNorm{l}{n} \cdot \beta^{q^{i+d}} \circ \FieldNorm{l}{m}, \fieldCharacter_l \right) &= \tau\left( \alpha \circ \FieldNorm{l}{n} \cdot \beta^{q^i} \circ \FieldNorm{l}{m}^{q^{d}}, \fieldCharacter_l \right)\\
	 &= \tau\left( \alpha \circ \FieldNorm{l}{n} \cdot \beta^{q^i} \circ \FieldNorm{l}{m}^{q^{an+bm}}, \fieldCharacter_l \right).
	 \end{align*}
	 Since the image of $\FieldNorm{l}{m}$ is in $\multiplicativegroup{\finiteFieldExtension{m}}$, we have that $\FieldNorm{l}{m}^{q^{bm}} = \FieldNorm{l}{m}$. Similarly, the image of $\FieldNorm{l}{n}$ is in $\multiplicativegroup{\finiteFieldExtension{n}}$, and therefore $\FieldNorm{l}{n}^{q^{an}} = \FieldNorm{l}{n}$. Hence we get
	 \begin{align*}
	 	\tau\left( \alpha \circ \FieldNorm{l}{n} \cdot \beta^{q^{i+d}} \circ \FieldNorm{l}{m}, \fieldCharacter_l \right) &= \tau\left( \alpha \circ \FieldNorm{l}{n}^{q^{an}} \cdot \beta^{q^i} \circ \FieldNorm{l}{m}^{q^{an}}, \fieldCharacter_l \right)\\
	 	&= \tau\left( \alpha \circ \FieldNorm{l}{n} \cdot \beta^{q^i} \circ \FieldNorm{l}{m}, \fieldCharacter_l \right),
	 \end{align*}
	 
\end{proof}

Next, we prove a property about Gauss sums of $\multiplicativegroup{\finiteFieldExtension{2m}}$. This result is needed for proving that the factors $\epsilon_0\left( \pi \times \sigma, \fieldCharacter \right)$ and $\gamma\left( \pi \times \sigma, \fieldCharacter \right)$ agree for $\pi \cong \sigma^{\vee}$ in \Cref{thm:relation-between-epsilon-factors-and-gamma-factors-cuspidal}.
\begin{proposition}\label{prop:gauss-sum-lemma} \label{prop:gauss-sums-even-field-extension-lemma}
	Let $\beta \in \charactergroup{2m}$ be a non-trivial multiplicative character $\beta \colon \multiplicativegroup{\finiteFieldExtension{2m}} \rightarrow \multiplicativegroup{\cComplex}$, such that $\beta$ is trivial on $\multiplicativegroup{\finiteFieldExtension{m}}$. Then $$\tau \left( \beta, \fieldCharacter_{2m}\right) = -q^m \beta^{-1}\left(z \right),$$
	where $z \in \multiplicativegroup{\finiteFieldExtension{2m}}$ satisfies $ z^{q^m - 1} = -1$.
\end{proposition}
\begin{proof}
	Write
	$$ \tau \left( \beta, \fieldCharacter_{2m}\right) = -\sum_{x \in \multiplicativegroup{\finiteFieldExtension{2m}}} \beta^{-1}\left(x\right) \fieldCharacter_{2m}\left(x\right).$$
	Since $\beta$ is trivial on $\multiplicativegroup{\finiteFieldExtension{m}}$, we may write
	$$-\tau \left( \beta, \fieldCharacter_{2m}\right) = \frac{1}{q^m - 1} \sum_{a \in \multiplicativegroup{\finiteFieldExtension{m}}} \sum_{x \in \multiplicativegroup{\finiteFieldExtension{2m}}} \beta^{-1}\left(a^{-1} x\right) \fieldCharacter_m \left( \FieldTrace_{\finiteFieldExtension{2m} \slash \finiteFieldExtension{m}}\left(x\right) \right).$$
	Changing variable $x = ay$, we obtain
	$$ -\tau \left( \beta, \fieldCharacter_{2m}\right) = \frac{1}{q^m - 1} \sum_{a \in \multiplicativegroup{\finiteFieldExtension{m}}} \sum_{y \in \multiplicativegroup{\finiteFieldExtension{2m}}} \beta^{-1}\left(y \right) \fieldCharacter_m \left( a \FieldTrace_{\finiteFieldExtension{2m} \slash \finiteFieldExtension{m}}\left(y\right) \right).$$
	We divide the sum into two sums according to whether the trace is zero or not:
	\begin{align*}
	-\tau \left( \beta, \fieldCharacter_{2m}\right) &= \frac{1}{q^m - 1}\sum_{\substack{y \in \multiplicativegroup{\finiteFieldExtension{2m}} \\
			\FieldTrace_{\finiteFieldExtension{2m} \slash \finiteFieldExtension{m}}\left(y\right) \ne 0}} \sum_{a \in \multiplicativegroup{\finiteFieldExtension{m}}} \beta^{-1}\left(y \right) \fieldCharacter_m \left( a \FieldTrace_{\finiteFieldExtension{2m} \slash \finiteFieldExtension{m}}\left(y\right) \right) \\
	&	+ \frac{1}{q^m - 1}
	\sum_{\substack{y \in \multiplicativegroup{\finiteFieldExtension{2m}} \\
			\FieldTrace_{\finiteFieldExtension{2m} \slash \finiteFieldExtension{m}}\left(y\right) = 0}} \sum_{a \in \multiplicativegroup{\finiteFieldExtension{m}}} \beta^{-1}\left(y \right) \fieldCharacter_m \left( a \FieldTrace_{\finiteFieldExtension{2m} \slash \finiteFieldExtension{m}}\left(y\right) \right)
	\end{align*}
	Regarding first sum, since $\FieldTrace_{\finiteFieldExtension{2m} \slash \finiteFieldExtension{m}}\left(y\right)$ is non-zero, we may change variables and get
	$$ \sum_{a \in \multiplicativegroup{\finiteFieldExtension{m}}} \fieldCharacter_m \left( a \FieldTrace_{\finiteFieldExtension{2m} \slash \finiteFieldExtension{m}}\left(y\right) \right) = -1.$$
	Hence we have,
	\begin{align*}
	-\tau \left( \beta, \fieldCharacter_{2m}\right) &= -\frac{1}{q^m - 1}\sum_{\substack{y \in \multiplicativegroup{\finiteFieldExtension{2m}} \\
			\FieldTrace_{\finiteFieldExtension{2m} \slash \finiteFieldExtension{m}}\left(y\right) \ne 0}}  \beta^{-1}\left(y \right) + \frac{1}{q^m - 1}
	\sum_{\substack{y \in \multiplicativegroup{\finiteFieldExtension{2m}} \\
			\FieldTrace_{\finiteFieldExtension{2m} \slash \finiteFieldExtension{m}}\left(y\right) = 0}} \sum_{a \in \multiplicativegroup{\finiteFieldExtension{m}}} \beta^{-1}\left(y \right)\\
	& = -\frac{1}{q^m - 1}\sum_{\substack{y \in \multiplicativegroup{\finiteFieldExtension{2m}} \\
			\FieldTrace_{\finiteFieldExtension{2m} \slash \finiteFieldExtension{m}}\left(y\right) \ne 0}}  \beta^{-1}\left(y \right) + 
	\sum_{\substack{y \in \multiplicativegroup{\finiteFieldExtension{2m}} \\
			\FieldTrace_{\finiteFieldExtension{2m} \slash \finiteFieldExtension{m}}\left(y\right) = 0}} \beta^{-1}\left(y \right).
	\end{align*}
	Since $\beta$ is a non-trivial multiplicative character, we have that $\sum_{y \in \multiplicativegroup{\finiteFieldExtension{2m}}} \beta^{-1}\left(y\right) = 0$, and therefore we can write
	\begin{align*}
	-\tau \left( \beta, \fieldCharacter_{2m}\right) &= \frac{q^m}{q^m - 1}\sum_{\substack{y \in \multiplicativegroup{\finiteFieldExtension{2m}} \\
			\FieldTrace_{\finiteFieldExtension{2m} \slash \finiteFieldExtension{m}}\left(y\right) = 0}}  \beta^{-1}\left(y \right).
	\end{align*}
	Since the trace map $\FieldTrace_{\finiteFieldExtension{2m} \slash \finiteFieldExtension{m}} \colon \finiteFieldExtension{2m} \rightarrow \finiteFieldExtension{m}$ is not injective, there exists $0 \ne z \in \finiteFieldExtension{2m}$, such that $\FieldTrace_{\finiteFieldExtension{2m} \slash \finiteFieldExtension{m}}\left(z\right) = z + z^{q^m} = 0$. If $y \in \multiplicativegroup{\finiteFieldExtension{2m}}$ satisfies $\FieldTrace_{\finiteFieldExtension{2m} \slash \finiteFieldExtension{m}} \left(y\right) = 0$, i.e., $y^{q^m} + y = 0$, then $y^{q^m - 1} = -1$, and then $$\left(\frac{y}{z} \right)^{q^m - 1} = 1,$$ i.e., $\frac{y}{z} \in \multiplicativegroup{\finiteFieldExtension{m}}$. Therefore $y = c \cdot z$, where $c \in \multiplicativegroup{\finiteFieldExtension{m}}$. Conversely, if $c \in \multiplicativegroup{\finiteFieldExtension{m}}$, then $\FieldTrace_{\finiteFieldExtension{2m} \slash \finiteFieldExtension{m} }\left(cz\right) = c \FieldTrace_{\finiteFieldExtension{2m} \slash \finiteFieldExtension{m}}\left(z\right) = 0$. Therefore, we have
	\begin{align*}
	-\tau \left( \beta, \fieldCharacter_{2m}\right) &= \frac{q^m}{q^m - 1}\sum_{c \in \multiplicativegroup{\finiteFieldExtension{m}}}  \beta^{-1}\left(cz \right) = q^m \beta^{-1}\left(z\right),
	\end{align*}
	where in the last step we used again the fact that $\beta$ is trivial on $\multiplicativegroup{\finiteFieldExtension{m}}$.
\end{proof}

\section{Polynomials with unitary roots}

In this appendix, we prove a property about polynomials with roots lying on the unit circle. We were not able to find a reference for this property in the literature.

We are interested in the following result.
\begin{proposition}\label{prop:unitary-polynomial-deformation-is-unitary}
	Let $$P\left(X\right) = \sum_{k=0}^n a_k X^k \in \cComplex\left[X\right]$$ be a polynomial, such that all of its roots lie on the unit circle. Then for any $-1 < \delta < 1$, the polynomial $$P_\delta\left(X\right) = \sum_{k=0}^n a_k \delta^{k\left(n-k\right)} X^k$$ also has all of its roots lying on the unit circle.
\end{proposition}

We first recall the Lee--Yang theorem \cite{lee1952statistical} \cite[Theorem 8.4]{borcea2009lee}. For a positive integer $n$, let $\left[n\right] = \left\{ 1, \dots, n\right\}$. For a complex number $z$, denote by $\conjugate{z}$ its complex conjugate.

\begin{theorem}[Lee--Yang]
	Let $A = \left(a_{ij}\right)_{i,j} \in \squareMatrix_n \left(\cComplex\right)$ be a hermitian matrix of size $n \times n$, i.e, $a_{ij} = \conjugate{a_{ji}}$. Suppose that $\abs{a_{ij}} \le 1$ for every $i,j$. Then the following polynomial has all of its roots lying on the unit circle: $$\operatorname{LY}_A\left(X\right) = \sum_{S \subseteq \left[n\right]} \left(\prod_{i \in S, j \notin S} a_{ij} \right) X^{\sizeof{S}},$$
	where the empty product is defined to be $1$.
\end{theorem}

We first show that certain polynomials satisfying the assumptions of \Cref{prop:unitary-polynomial-deformation-is-unitary} can be represented as polynomials occurring in the Lee--Yang theorem. We thank Fedor Petrov for providing the proof of this lemma on MathOverflow.
\begin{lemma}\label{lem:unitary-polynomial-is-a-lee-young-polynomial}
	Suppose that $$P\left(X\right) = \sum_{k=0}^n a_k X^k$$ is a polynomial all of whose roots lie on the unit circle, such that $a_0 = a_n = 1$. Then there exists a hermitian matrix $A = \left(a_{ij}\right) \in \squareMatrix_n \left(\cComplex\right)$, such that $\abs{a_{ij}} = 1$ for every $i,j$, with $P\left(X\right) = \leeYang_A\left(X\right)$.
\end{lemma}
\begin{proof}
	Write $$ P\left(X\right) = \prod_{i=1}^n \left(X + z_i\right),$$ where $z_i \in \cComplex$ with $\abs{z_i} = 1$ and $\prod_{i=1}^n z_i = 1$. Consider the following hermitian matrix $A = \left(a_{ij}\right) \in \squareMatrix_n\left(\cComplex\right)$ with $\abs{a_{ij}} = 1$ for every $i$ and $j$:
	$$a_{ij} = \begin{dcases}
	1 & i,j \ne n \text{ or } i=j=n,\\
	z_i & i \ne n \text{ and } j = n,\\
	\conjugate{z_j} & i = n \text{ and } j \ne n.
	\end{dcases}$$
	
	We claim that $\leeYang_A\left(X\right) = P\left(X\right)$. Let  $S \subseteq \left[n\right]$. Note that $\prod_{i \in S, j \notin S} a_{ij} = 1$ if $S = \left[n\right]$ or if $S = \emptyset$, so we may assume $\emptyset \subsetneqq S \subsetneqq \left[n\right]$.
	We compute $\prod_{i \in S,j \notin S} a_{ij}$. There are two cases to consider: $n \in S$ and $n \notin S$.
	
	If $n \notin S$, then $$\prod_{i \in S, j \notin S} a_{ij} = \prod_{i \in S}a_{i n} = \prod_{i \in S} z_i.$$	
	If $n \in S$, then $$\prod_{i \in S, j \notin S}a_{ij} = \prod_{j \notin S} a_{n j} = \prod_{j \notin S} \conjugate{z_j} = \prod_{j \notin S} {z_j^{-1}} = \prod_{j \in S}z_j,$$
	where we used the facts that $z_j \conjugate{z_j} = 1$ and that $\prod_{j = 1}^n z_j = 1$.
	
	Therefore, in all cases, we obtain $$\prod_{i \in S, j \notin S}a_{ij} = \prod_{j \in S}z_j.$$ Therefore, $$\leeYang_A\left(X\right) = \sum_{S \subseteq \left[n\right]} \left(\prod_{j \in S}z_j\right) X^{\sizeof{S}} = \prod_{i=1}^n \left(X + z_i\right) = P\left(X\right),$$
	as required.
\end{proof}

We are now ready to prove \Cref{prop:unitary-polynomial-deformation-is-unitary}.

\begin{proof}
	Let $$P\left(X\right) = \sum_{k=0}^{n} a_k X^k$$ be a polynomial all of whose roots are of absolute value $1$, and let $-1 < \delta < 1$. Without loss of generality, we may assume $a_n = 1$, as for $c \in \cComplex$, $c P_{\delta}\left(X\right) = \left(c P\right)_\delta\left(X\right)$. Since all of the roots of $P\left(X\right)$ are of absolute value $1$, we have that $\abs{a_0} = 1$. Let $\zeta \in \cComplex$ be such that $\zeta^n = a_0$. Then $\abs{\zeta}=1$ and $Q\left(X\right) = \zeta^{-n} P\left( \zeta X \right)$ is a monic polynomial of degree $n$ all of whose roots are of absolute value $1$, and its constant coefficient is $1$. By \Cref{lem:unitary-polynomial-is-a-lee-young-polynomial}, we have that there exists a hermitian matrix $A = \left(a_{ij}\right)$, with $\abs{a_{ij}} = 1$, such that $Q\left(X\right) = \leeYang_A \left(X\right)$. We have that $\delta A$ is still a hermitian matrix all of which entries are of absolute value not greater than $1$. We have $Q_{\delta}\left(X\right) = \leeYang_{\delta A} \left(X\right)$. By the Lee--Yang theorem, the roots of $Q_{\delta}\left(X\right)$ are all of absolute value $1$. Since $P_{\delta}\left(X\right) = \zeta^n Q_{\delta}\left(\zeta^{-1} X\right)$, and since $\zeta$ is of absolute value $1$, we have that the roots of $P_{\delta}\left(X\right)$ are all of absolute value $1$.
\end{proof}

\bibliographystyle{abbrv}
\bibliography{references}
\end{document}